\documentclass[11pt,reqno]{amsart}
\usepackage{amsthm}
\usepackage{amssymb}
\usepackage{latexsym}
\usepackage{multicol}
\usepackage{verbatim,enumerate}
\usepackage{hyperref}
\usepackage{amsmath, amscd}
\usepackage[all,cmtip]{xy}
\usepackage{soul}

\advance\textwidth by 1.2in \advance\oddsidemargin by -.6in \advance\evensidemargin by -.6in
\usepackage{tikz}
\usetikzlibrary{decorations.markings}
\tikzstyle{vertex}=[ circle, fill, draw, inner sep=0pt, minimum size=4pt,]
\tikzstyle{edge}= [thick]

\newtheorem*{cor}{Corollary}
\newtheorem*{lem}{Lemma}
\newtheorem*{prop}{Proposition}

\theoremstyle{definition} 
\theoremstyle{definition}
\newtheorem{thm}{Theorem}
\newtheorem*{thm*}{Theorem}

\newtheorem*{rem}{Remark}

\newenvironment{pf}{\proof}{\endproof}
\newcounter{cnt}
\newenvironment{enumerit}{\begin{list}{{\hfill\rm(\roman{cnt})\hfill}}{%
\settowidth{\labelwidth}{{\rm(iv)}}\leftmargin=\labelwidth%
\advance\leftmargin by \labelsep\rightmargin=0pt\usecounter{cnt}}}{\end{list}} \makeatletter
\def\mydggeometry{\makeatletter\dg@YGRID=1\dg@XGRID=20\unitlength=0.003pt\makeatother}
\makeatother \theoremstyle{remark}

\numberwithin{equation}{section}

 \DeclareMathOperator{\Ht}{ht}

\begin{document}

\newcommand{\thmref}[1]{Theorem~\ref{#1}}
\newcommand{\secref}[1]{Section~\ref{#1}}
\newcommand{\lemref}[1]{Lemma~\ref{#1}}
\newcommand{\propref}[1]{Proposition~\ref{#1}}
\newcommand{\corref}[1]{Corollary~\ref{#1}}
\newcommand{\remref}[1]{Remark~\ref{#1}}
\newcommand{\defref}[1]{Definition~\ref{#1}}
\newcommand{\er}[1]{(\ref{#1})}
\newcommand{\id}{\operatorname{id}}
\newcommand{\ord}{\operatorname{\emph{ord}}}
\newcommand{\sgn}{\operatorname{sgn}}
\newcommand{\wt}{\operatorname{wt}}
\newcommand{\tensor}{\otimes}
\newcommand{\from}{\leftarrow}
\newcommand{\nc}{\newcommand}
\newcommand{\rnc}{\renewcommand}
\newcommand{\dist}{\operatorname{dist}}
\newcommand{\qbinom}[2]{\genfrac[]{0pt}0{#1}{#2}}
\nc{\cal}{\mathcal} \nc{\goth}{\mathfrak} \rnc{\bold}{\mathbf}
\renewcommand{\frak}{\mathfrak}
\newcommand{\supp}{\operatorname{supp}}
\newcommand{\Irr}{\operatorname{Irr}}
\newcommand{\psym}{\mathcal{P}^+_{K,n}}
\newcommand{\psyml}{\mathcal{P}^+_{K,\lambda}}
\newcommand{\psymt}{\mathcal{P}^+_{2,\lambda}}
\renewcommand{\Bbb}{\mathbb}
\nc\bomega{{\mbox{\boldmath $\omega$}}} \nc\bpsi{{\mbox{\boldmath $\Psi$}}}
 \nc\balpha{{\mbox{\boldmath $\alpha$}}}
 \nc\bbeta{{\mbox{\boldmath $\beta$}}}
 \nc\bpi{{\mbox{\boldmath $\pi$}}}
  \nc\bpis{{\mbox{\boldmath \scriptsize$\pi$}}}
 \nc\bullets{{\mbox{\scriptsize $\bullet$}}}
 
  \nc\bvarpi{{\mbox{\boldmath $\varpi$}}}

\nc\bepsilon{{\mbox{\boldmath $\epsilon$}}}
  \nc\bomegas{{\mbox{\boldmath\scriptsize $\omega$}}}
  \nc\bepsilons{{\mbox{\boldmath \scriptsize$\epsilon$}}}

  \nc\bxi{{\mbox{\boldmath $\xi$}}}
\nc\bmu{{\mbox{\boldmath $\mu$}}} \nc\bcN{{\mbox{\boldmath $\cal{N}$}}} \nc\bcm{{\mbox{\boldmath $\cal{M}$}}} \nc\blambda{{\mbox{\boldmath
$\lambda$}}}

\newcommand{\Tmn}{\bold{T}_{\lambda^1, \lambda^2}^{\nu}}

\newcommand{\lie}[1]{\mathfrak{#1}}
\newcommand{\ol}[1]{\overline{#1}}
\makeatletter
\def\section{\def\@secnumfont{\mdseries}\@startsection{section}{1}%
  \z@{.7\linespacing\@plus\linespacing}{.5\linespacing}%
  {\normalfont\scshape\centering}}
\def\subsection{\def\@secnumfont{\bfseries}\@startsection{subsection}{2}%
  {\parindent}{.5\linespacing\@plus.7\linespacing}{-.5em}%
  {\normalfont\bfseries}}
\makeatother
\def\subl#1{\subsection{}\label{#1}}
 \nc{\Hom}{\operatorname{Hom}}
  \nc{\mode}{\operatorname{mod}}
\nc{\End}{\operatorname{End}} \nc{\wh}[1]{\widehat{#1}} \nc{\Ext}{\operatorname{Ext}}
 \nc{\ch}{\operatorname{ch}} \nc{\ev}{\operatorname{ev}}
\nc{\Ob}{\operatorname{Ob}} \nc{\soc}{\operatorname{soc}} \nc{\rad}{\operatorname{rad}} \nc{\head}{\operatorname{head}}
\def\Im{\operatorname{Im}}
\def\gr{\operatorname{gr}}
\def\mult{\operatorname{mult}}
\def\Max{\operatorname{Max}}
\def\ann{\operatorname{Ann}}
\def\sym{\operatorname{sym}}
\def\loc{\operatorname{loc}}
\def\Res{\operatorname{\br^\lambda_A}}
\def\und{\underline}
\def\Lietg{$A_k(\lie{g})(\bs_\xiigma,r)$}
\def\res{\operatorname{res}}

 \nc{\Cal}{\cal} \nc{\Xp}[1]{X^+(#1)} \nc{\Xm}[1]{X^-(#1)}
\nc{\on}{\operatorname} \nc{\Z}{{\bold Z}} \nc{\J}{{\cal J}} \nc{\C}{{\bold C}} \nc{\Q}{{\bold Q}}
\renewcommand{\P}{{\cal P}}
\nc{\N}{{\Bbb N}} \nc\boa{\bold a} \nc\bob{\bold b} \nc\boc{\bold c} \nc\bod{\bold d} \nc\boe{\bold e} \nc\bof{\bold f} \nc\bog{\bold g}
\nc\boh{\bold h} \nc\boi{\bold i} \nc\boj{\bold j} \nc\bok{\bold k} \nc\bol{\bold l} \nc\bom{\bold m} \nc\bon{\bold n} \nc\boo{\bold o}
\nc\bop{\bold p} \nc\boq{\bold q} \nc\bor{\bold r} \nc\bos{\bold s} \nc\boT{\bold t} \nc\boF{\bold F} \nc\bou{\bold u} \nc\bov{\bold v}
\nc\bow{\bold w} \nc\boz{\bold z} \nc\boy{\bold y} \nc\ba{\bold A} \nc\bb{\bold B} \nc\bc{\mathbb C} \nc\bd{\bold D} \nc\be{\bold E} \nc\bg{\bold
G} \nc\bh{\bold H} \nc\bi{\bold I} \nc\bj{\bold J} \nc\bk{\bold K} \nc\bl{\bold L} \nc\bm{\bold M}  \nc\bo{\bold O} \nc\bp{\bold
P} \nc\bq{\bold Q} \nc\br{\bold R} \nc\bs{\bold S} \nc\bt{\bold T} \nc\bu{\bold U} \nc\bv{\bold V} \nc\bw{\bold W} \nc\bx{\bold
x} \nc\KR{\bold{KR}} \nc\rk{\bold{rk}} \nc\het{\text{ht }}
\nc\bz{\mathbb Z}
\nc\bn{\mathbb N}

\nc\toa{\tilde a} \nc\tob{\tilde b} \nc\toc{\tilde c} \nc\tod{\tilde d} \nc\toe{\tilde e} \nc\tof{\tilde f} \nc\tog{\tilde g} \nc\toh{\tilde h}
\nc\toi{\tilde i} \nc\toj{\tilde j} \nc\tok{\tilde k} \nc\tol{\tilde l} \nc\tom{\tilde m} \nc\ton{\tilde n} \nc\too{\tilde o} \nc\toq{\tilde q}
\nc\tor{\tilde r} \nc\tos{\tilde s} \nc\toT{\tilde t} \nc\tou{\tilde u} \nc\tov{\tilde v} \nc\tow{\tilde w} \nc\toz{\tilde z} \nc\woi{w_{\omega_i}}
\nc\chara{\operatorname{Char}}
\title[Tensor products  and $q$-characters of HL-modules . ]
{Tensor products and $q$-characters  of HL-modules\\   and monoidal categorifications}
\author[Brito\,\, and Chari]{Matheus Brito and Vyjayanthi Chari}
\address{Departamento de Matem\'atica, UFPR, Curitiba - PR - Brazil, 81530-015.}\thanks{MB was partially 
supported by CNPq, grant 205281/2014-1}
\email{mbrito@ufpr.br}
\address{Department of Mathematics\\ 
  University of California, Riverside\\ 
  900 University Ave., Riverside, CA 92521} \thanks{VC was partially supported by DMS	1719357}
\email{chari@math.ucr.edu}\maketitle
%
%
\begin{abstract} We study certain monoidal subcategories  (introduced by David Hernandez and Bernard Leclerc)  of finite--dimensional representations of a quantum affine algebra of type $A$.  We
classify the set of prime representations in these subcategories and give necessary and sufficient conditions for a tensor product of two prime representations to be irreducible. In the case of a reducible tensor product we 
describe the prime decomposition of the simple factors.  As a consequence we prove that these subcategories are monoidal categorifications of a cluster algebra of type A with coefficients. 
    
\end{abstract}
\section*{Introduction}

The study of the category $\cal F$ finite--dimensional representations of a quantum affine algebra  goes back nearly thirty years and continues to be of significant interest. The irreducible objects in this category are indexed by elements of a free abelian monoid (denoted $\cal P^+$) with generators $\bomega_{i,a}$ where $i$ varies over the index set for the simple roots and $a$ varies over non--zero elements of the field of rational functions in a variable $q$. The category is not semisimple and there are many interesting indecomposable objects in it.  In recent years, there has been new insight in the study of $\cal F$  coming from  connections with cluster algebras through the work of \cite{HL1}, \cite{HL2}, \cite{Nak} and also from KLR algebras through the work of \cite{KKKO, KKKO2}. 

The category $\cal F$ is a monoidal tensor category and  an interesting feature is  that a tensor product of generic  simple objects is simple.  An obviously related notion is that of a prime simple object; this is one which cannot be written in a non-trivial way as a tensor product of objects of $\cal F$. 
An open and very difficult question is the following: classify prime simple objects in $\cal F$ and describe the  factorization of an arbitrary  simple object as a tensor product of primes. 
The answer to this question for $\lie{sl}_2$ was given in \cite{CPqa} where it was also proved that the factorization was unique. In higher rank the  question along with that of uniqueness remains unanswered. However, in \cite{H5} and\cite{H6} an important result was established which greatly simplifies the problem by reducing it to  following: give a necessary and sufficient condition for the tensor product of a  pair of prime simple objects to be simple.

In this paper we focus on this question  for 
 certain  subcategories  of  $\cal F$ associated with  quantum affine $\lie{sl}_{n+1}$.
These subcategories  were introduced by  David Hernandez and Bernard Leclerc (\cite{HL1}, \cite{HL2}) and the definition has its roots in the theory of cluster algebras. The remarkable insight was that prime representations were analogous to cluster variables and the irreducibility of a tensor product of prime objects was analogous to the idea of two  elements belonging to the same cluster. The role of the quiver in the theory of cluster algebras is played by the height function; 
a height function (of type $A_n$) is a function  $\xi: [1,n]\to \bz$ satisfying the condition $|\xi(i)-\xi(i+1)|=1$ for $1\le i\le n-1$. Define  $\cal P^+_\xi$ to  be the submonoid of $\cal P^+$ generated by elements $\bomega_{i,q^{\xi(i)\pm 1 }}$ and 
let   $\cal F_\xi$ be  the full subcategory of $\cal F$ consisting of objects whose Jordan--Holder constituents  are indexed by elements of $\cal P^+_\xi$. It was proved in \cite{HL2} that $\cal F_\xi$ is a monoidal tensor category and  we let $\cal K_0(\cal F_\xi)$ be the Grothendieck ring of $\cal F_\xi$. In the case when $\xi$ is the bipartite height function, i.e,  $\xi(i-1)=\xi(i+1)$ for $2\le i\le n-2$ or the monotonic function $\xi(i)=i$ they showed that $\cal K_0(\cal F_\xi)$ is isomorphic to a cluster algebra with coefficients of type $A$. 

In this paper we prove the result for all height functions of type $A$  by  representation theoretic methods. We define a subset $\bp\bor_\xi$ of $\cal P^+_\xi$ such that the  corresponding irreducible representations (which we call  HL-modules)  are  prime. 
Working  entirely in  $\cal F_\xi$ we 
  show that  the  HL-modules are  precisely all the prime objects in this category.  To do this,  we  establish  necessary and sufficient conditions for a tensor product of HL-modules to be irreducible. In the case when the tensor product is reducible we describe the Jordan-Holder constituents and their factorization as a tensor product of HL-modules.
  
  The connection with cluster algebras is then made as follows. We define a   the   quiver $Q_\xi$ associated with $\xi$; since we are working in the general case the quiver we use is a mutation of the quivers  in \cite{HL1} and  \cite{HL2}. This mutation allows us to map a non--frozen variable in the  initial seed of the cluster algebra to  the class of the irreducible module corresponding to either  $\bomega_{i,\xi(i)+ 1}$ or $\bomega_{i,\xi(i)-1}$. The first mutation at any element of the initial seed is easily described; however is not necessarily of the form $\bomega_{i,\xi(i)\pm 1}$. Our tensor product formulae now allow us to prove the existence of an algebra 
  isomorphism between the cluster algebra with $n$ frozen variables and $\cal K_0(\cal F_\xi)$. The isomorphism maps a cluster variable to an HL-module and we identify this module explicitly.    We also  show that the isomorphism maps cluster monomials to simple tensor products of HL-modules.
  As a consequence of this result we give an alternate proof for the product of a pair of cluster variables to be a cluster monomial; equivalently we give an alternate proof of the criterion for  a pair of roots to be compatible. 
In Proposition \ref{closedform} we give a closed formula for a cluster variable in terms of the original seed. In terms of representation theory this can be interpreted as giving  a $q$-character formula for the prime representations in $\cal F_\xi$. It is useful to remark here that other explicit formulae for cluster variables can be found in the literature see for instance, \cite{BP}, \cite{CC}, \cite{Dup9}, \cite{FK10}. Not all these papers deal with  frozen variables and even those that do impose conditions on the frozen variables which are not satisfied by the quivers considered in this paper. The role of the frozen variable in the  connection with representation theory is important and motivates our formulae.

\medskip 
The paper is organized as follows. In Section \ref{main} we recall the definition of the height function $\xi$ and introduce  the associated quiver $Q_\xi$.  We then  
state and prove our main result modulo the key Propositions \ref{clusterind}, \ref{clusterrep} and \ref{clustermon}.  In Section \ref{proofofclusterind}  we prove Proposition \ref{clusterind} which gives a recursive formula for a cluster variable. This is done by a simple analysis of the quiver obtained by mutating at successive nodes. The answer we obtain is in a form which is well adapted to the representation theory of quantum affine algebras and can be viewed as an analog of Pieri's rule in classical representation theory.
We then solve the recursion to  give a closed formula for the cluster variable in terms of the initial cluster which includes  the frozen variables. In Sections \ref{irrtp}, \ref{reducible} and \ref{starti} we provide sufficient and necessary conditions,  for the tensor product of two HL-modules to be irreducible. We also analyze the Jordan-Holder series of a reducible tensor product of HL-modules.  The proof of  Propositions  \ref{clusterrep} and \ref{clustermon} can be found in  Section \ref{reducible}. 

\medskip

\noindent {\bf Acknowledgements.} {\em MB is grateful to the Department of Mathematics, UCR, for
their hospitality during a visit when part of this research was carried out. He also thanks  David Hernandez for supporting his visit to Paris 7 and many helpful discussions. VC thanks David Hernandez,  Bernard Leclerc and Salvatore Stella for  helpful conversations}.

\section{The main results} \label{main}
Throughout the paper  we denote by $\bc$, $\bz$, $\bz_+$ and  $\bn$  the set of complex numbers,  integers, non--negative and positive  integers respectively. For $i,j\in \bz_+$ with $i\le j$ we let $[i,j]=\{i,i+1,\cdots, j\}$.
Given a commutative  ring $A$ we denote by  $A[q]$ (resp. $A(q)$)   the ring (resp. quotient field) of polynomials  in an indeterminate $q$ with coefficients in $A$.
\subsection{ The cluster algebra $\cal A(\bx, Q_\xi)$}\hfill

Let $\xi: [1,n]\to \bz$ be a height function; namely a function which  satisfies the conditions $$ \ |\xi(i)-\xi(i-1)|= 1,\ \ 2\le i\le n.$$ 
It will be convenient  to extend $\xi$ to $[0,n+1]$ by setting $\xi(0)=\xi(2)$ and $\xi(n-1)=\xi(n+1)$.
\begin{rem} Although  trivial, it is  useful to note that $\{\xi(i+ 1),\xi(i-1)\}\subset\{\xi(i)+ 1, \xi(i)-1
\}$ and that the inclusion can be strict.
\end{rem}
\medskip

\noindent For $i\in [1,n-2]$,  let $i_\diamond\in[i,n]$  be minimal such that $\xi(i_\diamond)=\xi(i_\diamond+2)$ and set $(n-1)_\diamond=(n-1)$ and  $n_\diamond=n$.  Let 
$Q_\xi$ be a quiver  with $2n$ vertices  labeled $\{1,\cdots,n, 1',\cdots, n'\}$ and  with the set of edges  given as  follows: \begin{itemize}
\item there are no edges between the primed vertices; in other words the vertices $\{1',\cdots, n'\}$ are frozen,
\item 
if  $1\le j\le n-1$ and  $\xi(j)=\xi(j+1)+1$, the edges at $j$ are: $$
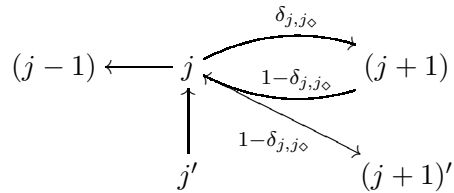
\xymatrix{
&  & (j-1) & j\ar[l] \ar[rrd]_{1-\delta_{j,j_\diamond}}\ar@/^1pc/[rr]^{\delta_{j,j_\diamond}} & & (j+1)\ar@/^1pc/[ll]_{1-\delta_{j,j_\diamond}}\\
&  & &j'\ar[u] & & (j+1)' &
}\\$$
and the reverse orientations if $\xi(j)=\xi(j+1)-1$, where $\delta_{j,j_\diamond}$ is the Kronecker delta function and  we adopt the convention that a labeled edge exists iff the label is one,\\
\item at the vertex $n$ we have  edges $(n-1)\to n\to n'$\  if $\xi(n-1)=\xi(n)+1$ and the reverse orientation otherwise.
\end{itemize}
 Clearly  $j$ is a sink or source of $Q_\xi$ (where we ignore the frozen vertices)  iff $j=1$ or $j=j_\diamond$.  Given $2\le j\le n$ let $2_\bullet=1$ and for $j>2$ let $j_\bullet$ be the maximal sink or source of $Q_\xi$ satisfying $j_\bullet<j$.

Fix a set   $\bx=\{x_1,\cdots, x_n, f_1,\cdots, f_n\}$ of algebraically independent variables and let 
$\cal A(\bx, Q_\xi)$ 
be the cluster algebra (with coefficients)  with initial seed $(\bx, Q_\xi)$. 
The definition of a cluster algebra is recalled briefly in Section \ref{cluster}; for the rest of this section  we shall freely use the language of cluster algebras. Since the principal unfrozen part of $Q_\xi$ is a quiver of type $A_n$,  the set of non-frozen cluster variables in  $\cal A(\bx, Q_\xi)$ are indexed by the set $\Phi_{\ge -1}$
of almost positive roots of a root system of type $A_n$. In other words if we  let $\{\alpha_i:1\le i\le n\}$, be  a set of simple roots for $A_n$ and  set $\alpha_{i,j}=\alpha_i+\cdots+\alpha_j$, $1\le i\le j\le n$, 
then $$\Phi_{\ge -1}=\{-\alpha_i,\  \alpha_{i,j}: 1\le i\le j\le n\},$$ and the cluster variables are denoted $$\{x_i:= x[-\alpha_i], \ x[\alpha_{i,j}],\ f_i:\ 1\le i\le j\le n\}.$$

\subsection{The category $\cal F_\xi$}\hfill

Let $\widehat\bu_q$ be the quantum loop algebra over $\bc(q)$  associated to $\lie{sl}_{n+1}$ and let $\cal F$ be the monoidal tensor category whose objects are  finite--dimensional representations of $\widehat\bu_q$.
Given a height function $\xi:[1,n]\to\bz$ we take
 $\cal P^+_\xi$  to  be the free abelian  monoid with generators $\{\bomega_{i,\xi(i)\pm 1}: i\in[1,n]\}.$ 
It is known that   $\cal P^+_\xi$ is the index set for a  (sub)-family of isomorphism classes of irreducible objects  of $\cal F$. We define $\cal F_\xi$ to be the full subcategory of $\cal F$ consisting of objects all of whose  Jordan--Holder constituents are indexed by elements of  $\cal P^+_\xi$. It was proved in \cite{HL2} that $\cal F_\xi$ is a monodial category and we let $\cal K_0(\cal F_\xi)$ be the corresponding Grothendieck ring. For $\bomega\in\cal P^+_\xi$ let  $[\bomega]\in\cal K_0(\cal F_\xi)$ be  the isomorphism class of the corresponding object in $\cal F_\xi$.
\begin{rem} It is important to keep in mind that the assignment $\bomega\to[\bomega]$ is not a morphism of monoids $\cal P^+_\xi\to\cal K_0(\cal F_\xi)$, i.e., $[\bomega][\bomega']$ is not always equal to $[\bomega\bomega']$. One of the goals of this paper is to determine a necessary and sufficient condition for equality to hold.
\end{rem}

\noindent
For $i\in[1,n]$ set $$\bof_i=\bomega_{i,\xi(i)+1}\bomega_{i,\xi(i)-1}.$$ 
If $1\le i<j\le n$, 
let   $i_2<\cdots<i_{k-1}$  be an ordered enumeration of the subset $\{p: i<p<j, \  \xi(p-1)=\xi(p+1)\}$,  $i_1=i$, $i_k=j$ and 
define an element $\bomega(i,j)\in\cal P^+_\xi$ by: \begin{gather*}
\bomega_\xi(i,j)=\bomega_{i_1,a_1}\cdots\bomega_{i_k,a_k},\end{gather*} where $a_1=\xi(i)\pm 1$ if $\xi(i+1)= \xi(i)\mp 1$ and $a_m=\xi(i_m)\pm 1\ \ {\rm{if}}\ \ \xi(i_{m})=\xi(i_m-1)\pm 1$, for $m\ge 2$.
Set $$\bp\bor_\xi=\{  \bomega_{i,\xi(i)\pm 1},\ \bomega(i,j) : 1\le i\leq j\le n, \ i\neq j \}.$$
Clearly the set $\bp\bor_\xi$ has the same cardinality as the set of unfrozen cluster variables in $\cal A(\bx, Q_\xi)$.
\medskip

\noindent Recall that an object of $\cal F$ is said to be prime if it cannot be written in a nontrivial way as a tensor product of objects of $\cal F$. The following is a special case of the main result  of  \cite{CMY}.
\begin{lem}\label{prime} The irreducible object of $\cal F_\xi$ associated to an element  $\bomega\in\bp\bor_\xi\cup\{\bof_i: 1\le i\le n\}$ is prime. \hfill\qedsymbol
\end{lem}

\subsection{Main Theorem}
Recall  that by definition $n=n_\diamond$ and $(n-1)=(n-1)_\diamond$ which in particular implies that $n_\bullet=n-1$. For $k\ge 2$, set   \begin{equation}\label{bark} \bar{ k}=(k+1)(1-\delta_{k,k_\diamond})+(k_\bullet+1)\delta_{k,k_\diamond}. \end{equation}
\begin{thm}\label{main}Let $\xi: [1,n]\to\bz$ be a height function. The assignment
 $$\iota(x_i)= [\bomega_{i, \xi(i+1)}],\ \ \ \iota(f_i)=[\bold f_i],$$ extends to an  isomorphism of
  rings
 $\iota: \cal A(\bx, Q_\xi)
 \to \cal K_0(\cal F_\xi)$ such that 
for $1\le i\le k\le n$, \begin{eqnarray}\label{iota1}
&\iota(x[\alpha_{i,i_\diamond}])&= [\bomega_{i,\xi(i+1)\pm 2}],\ \ \ \  \xi(i)=\xi(i+1)\pm 1\\ \notag\\ \label{iota2} &\iota(x[\alpha_{i,k}])&=[\bomega(i,\bar{k})], \ \ k\ne i_\diamond, \ \\ \notag \\ &
\iota(f_px[\alpha])&=[\bof_p\bomega]\ \  p\in[1,n],\ \ \alpha\in  \Phi_{\ge -1},\ \ \ [\bomega]=\iota(x[\alpha]).
\end{eqnarray}
In particular $\iota$ maps  cluster variable to a prime object of $\cal F_\xi$. Moreover,
if $\beta_1,\beta_2\in \Phi_{\ge -1}$ are such that $x[\beta_1]x[\beta_2]$ is  a cluster monomial then  $$[\bomega_1\bomega_2]=[\bomega_1][\bomega_2],\ \ [\bomega_s]=\iota(x[\beta_s]),\ \ s=1,2 .$$ 
\end{thm}

\begin{cor} The homomorphism $\iota$ sends a cluster monomial to the equivalence class of an irreducible object of $\cal F_\xi$. In particular, $\cal F_\xi$ is a monoidal categorification of $\cal A(\bx,\xi)$.
\end{cor}
\noindent {\em{Proof of Corollary.}} Let $x[\beta_1]\cdots x[\beta_r]$ be  a cluster monomial for some $\beta_1,\cdots,\beta_r\in \Phi_{\ge -1}$ and set $[\bomega_i]=\iota(x[\beta_i])$ for $1\le i\le r$. Then the  pairs $x[\beta_j]x[\beta_p]$, $1\le j\ne  p\le r$ are cluster monomials and hence using the Theorem \ref{main} we have $\iota(x[\beta_j] x[\beta_p])=[\bomega_j\bomega_p]$ for $1\le j\ne p\le r$.
It follows from the  main result of  \cite{H} (see section 3 of this paper for the statement)  that $\iota(x[\beta_1]\cdots x[\beta_r])=[\bomega_1\cdots\bomega_r]$  and the corollary is established.\hfill\qedsymbol

\begin{rem}  Suppose that $\xi$ satisfies  $\xi(i-1)=\xi(i+1)$ for all $1\le i\le n$ or that   $\xi(j)=\xi(i)+(j-i)$ for all $1\le i\le j\le n$. In these two cases the existence of $\iota$ was established in \cite{HL1},\cite{HL2} by very different methods. As was noted in \cite{HL2} the categories $\cal F_\xi$ are not necessarily equivalent for different height functions.  \end{rem}
\medskip

\subsection{} In  Theorem \ref{irredcrit} of this paper we give   conditions for the equality $[\bpi\bpi']=[\bpi][\bpi']$ when $\bpi,\bpi'\in \bp\bor_\xi$ to hold in $\cal K_0(\cal F_\xi)$. The translation to  the language of cluster algebra  gives the conditions for  describing when two roots are compatible. Thus our theorem gives a proof of the following assertion (compare with the description in Section 10.2.3 of \cite{HL2} where a similar description in the case of the bipartite height function).

\noindent Assume that $i\le j$,  $k\le \ell$ and $i\le k$.
If $j\ne i_\diamond$ the roots  $\alpha_{i,j}$, $\alpha_{k,\ell}$ are compatible iff:

\noindent $\bullet $ $k=i\ \ {\rm or }\   k>j+1$,\\ \\
$\bullet$ $j=j_\diamond$ and  $j_\bullet+1\leq  k\le j,$\\ \\ 
$\bullet$ $\ell\ne k_\diamond$ and either $\bar j=\bar\ell$ or \begin{gather*}i<k<\bar j< \bar\ell,  \ \ {\rm {and}}\ \#\{k\le m<\bar j-1: m=m_\diamond\}\in 2\bz_++1,\ {\rm{or}}\\
 \ i<k< \bar\ell<\bar j,  \ \ {\rm and }\ \#\{k\le m<\bar\ell-1: m=m_\diamond\}\in 2\bz_+.\end{gather*}

\noindent The roots $\alpha_{i,i_\diamond}$ and $\alpha_{k,\ell}$ with $i\le k$ are compatible iff :
\begin{gather*}  \ \ k_\bullet\ne k-1,\ \ {\rm{or}}\ \  (k-1)_\bullet\ge i
 \ {\rm{or}}\ \ 
\ell\ne k_\diamond \ \ {\rm{and}}\ \ i=k.
\end{gather*}

\noindent The roots $-\alpha_{i}$ and $\alpha_{k,\ell}$  are in the same cluster iff either $k>i$ or $\ell<i$.

\medskip

In Theorem \ref{t:reduc} we write down the Jordan-Holder series for  a reducible tensor product of objects.  This amounts to writing down all the non-trivial exchange relations for cluster variables including the frozen variables and is not hard to do using the analysis above.

\subsection{} The proof of the theorem involves three principal steps. For $1\le j\le n$, set $$ d_j=\delta_{j, j_\diamond}=\delta_{\xi(j),\xi(j+2)}.$$
The first step is the following proposition which gives a recursive formula for the cluster variables.
We adopt the convention that $\alpha_{i,m}=\alpha_m,\   m\le i$.

\begin{prop} \label{clusterind}  
For $1\le i< j\le n$ we have 
$$x_ix[\alpha_i]=f_ix_{i+1}^{1-d_i}+f_{i+1}^{1-d_i}x_{i-1}x_{i+1}^{d_i},$$
\begin{eqnarray*}&x_jx[\alpha_{i,j}]& = f_j^{d_{j-1}}x[\alpha_{i,j-1}]x_{j+1}^{1-d_j}+\\&&+ f_{j+1}^{1-d_j}x_{j+1}^{d_j} \left((\delta_{i_\bullet,j_\bullet}+\delta_{i,j_\bullet})f_i^{\delta_{i,j_\bullet}}x_{i-1}^{1-\delta_{i,j_\bullet}}+(1-\delta_{i_\bullet,j_\bullet}-\delta_{i,j_\bullet}) f_{j_\bullet}^{d_{j_\bullet-1}} x[\alpha_{i,j_\bullet-1}]\right).
\end{eqnarray*}
\end{prop}
The 
proof of this proposition is  in Section \ref{proofofclusterind} where we also give a closed formula for $x[\alpha_{i,j}]$ as a Laurent polynomial in the variables $\{x_1,\cdots,x_n,f_1,\cdots, f_n\}$.

\subsection{} The second step in the proof of the theorem is the following.  We adopt the convention that we take $\bomega_{i,\xi(i+1)+2}$ if $\xi(i)=\xi(i+1)+1$ and we take $\bomega_{i,\xi(i+1)-2}$ if $\xi(i)=\xi(i+1)-1$.
\begin{prop}\label{clusterrep}
 The following equalities hold in $\cal K_0(\cal F_\xi)$ for $1\le i\le j\le n$.
 \begin{enumerit}
\item[(i)] We have 
$$[\bomega_{i,\xi(i+1)}][\bomega(i,i+1)]^{1-d_i}[\bomega_{i,\xi(i+1)\pm 2}]^{d_i}=   [\bof_i]\ [\bomega_{i+1, \xi(i+2)}]^{1-d_i}  + [\bof_{i+1}]^{1-d_i} [\bomega_{i+1,\xi(i+2)}]^{d_i}\  [\bomega_{i-1,\xi(i)}]\ .
$$
\item[(ii)] If $j_\bullet\le i<j$ then

  \begin{eqnarray*}& [\bomega_{j, \xi(j+1)}][\bomega(i,\bar j)]^{1-\delta_{j,i_\diamond}}[\bomega_{i,\xi(i+1)\pm 2}]^{\delta_{j,i_\diamond}}&= [\bof_j]^{d_{j-1}}\ 
[\bomega(i,j)]^ {1-d_{j-1}}  [\bomega_{i, \xi(i+1)\pm 2}]^{d_{j-1}} [\bomega_{j+1, \xi(j+2)}]^{1-d_j}\\ \\ && +   [\bof_{j+1}]^{1-d_j} [\bof_i]^{\delta_{i,j_\bullet}} [\bomega_{j+1, \xi(j+2)}]^{d_j} [\bomega_{i-1, \xi(i)}]^{1-\delta_{i,j_\bullet}}\end{eqnarray*}
	\item[(iii)] If $i<j_\bullet$ choose $z\in\{\xi(i)+1,\xi(i)-1\}$ so that $\bomega_{i,z}^{-1}\bomega(i,j_\bullet)\in\cal P^+_\xi$ and set $k=(j_\bullet)_\bullet$.  Then,
	\begin{eqnarray*}[\bomega_{j,\xi(j+1)}] [\bomega(i,\bar j)]= [\bof_{j}]^{d_{j-1}}\ [\bomega_{j+1, \xi(j+2)}]^{1-d_j}\ [\bomega(i,\overline{ j-1})]^{1-\delta_{(j-1)_\bullet, i_\bullet}}\  [\bomega_{i,z}]^{\delta_{(j-1)_\bullet,i_\bullet}}\\ \\ + [\bof_{j+1}]^{1-d_j}\ [\bof_{j_\bullet}]^{d_{j_\bullet-1}} [\bomega_{j+1, \xi(j+2)}]^{d_j} \ [\bomega(i,\overline{j_\bullet-1} )]^{1-\delta_{i_\bullet, k_\bullet}d_{j_\bullet-1}}\   [\bomega_{i,z}] ^{\delta_{i_\bullet, k_\bullet}d_{j_\bullet-1}}\ . \end{eqnarray*}
\end{enumerit}
\end{prop}

The proof of this proposition can be found in Section \ref{reducible}. 

 \subsection{} Proposition \ref{clusterind} and Proposition \ref{clusterrep} are enough to establish the existence of $\iota$ and to identify the image of a cluster variable. 
 The third step needed to establish the theorem 
is to show that $\iota$ maps a cluster monomial to the isomorphism class of an irreducible representation. To do this we will need the following result.
\begin{prop}\label{clustermon} Let $\bomega,\bomega'\in\bp\bor_\xi$. Then either $[\bomega][\bomega']=[\bomega\bomega']$ or $[\bomega][\bomega']=[\bomega_1]+[\bomega_2]$ where $[\bomega_1]$ and $[\bomega_2]$ are the images under $\iota$  of cluster monomials.
\end{prop}
A much more precise statement can be found in Theorem \ref{irredcrit} and Theorem \ref{t:reduc} in Sections \ref{irrtp} and \ref{reducible}. 
In the rest of this section we assume Proposition \ref{clusterind}, Proposition \ref{clusterrep}, Proposition \ref{clustermon} and prove Theorem \ref{main}. 
 \subsection{Existence of $\iota$}\label{iotaexists} Recall \cite{BFZ} that an element of $\cal A(\bx,Q_\xi)$ is said to be a standard monomial  if it is a monomial in the elements $\{x_i,x[\alpha_i]: i\in[1,n]\}$  and does not involve any product of the form $x_ix[\alpha_i]$, $i\in[1,n]$. It was
proved in \cite{BFZ} that  standard
monomials are a $\bz[f_i: i\in I]$--basis of $\cal A(\bx,\xi)$. 
\medskip

\noindent
On the other hand consider the quotient of the  polynomial ring (with integer coefficients) in variables $X_i, X[\alpha_i],F_i$, $i\in[1,n]$ subject to the first relation in Proposition \ref{clusterind}. It is not hard to show that this  ring is  the $
\bz[F_i: i\in I]$ span of
 monomials in $X_i, X[\alpha_i]$, $i\in[1,n]$
  which do not involve products of $X_i X[\alpha_i] $ for any $i\in[1,n]$. 
  It follows that 
 $\cal A(\bold x, Q_\xi)$ is is isomorphic to this quotient
 (compare with \cite[Lemma 4.4]{HL1}).
 \medskip
 \noindent 
Using  Proposition \ref{clusterrep} (i)  we have
\begin{eqnarray*}[\bomega_{i,\ \xi(i+1)}][\bomega(i,i+1)]^{1-d_i}[\bomega_{i,\xi(i+1)\pm 2}]^{d_i}= [\bof_i] [\bomega_{i+1, \xi(i+2)}]^{1-d_i} + [\bof_{i+1}]^{1-d_i}[\bomega_{i+1, \xi(i+2)}]^{d_i}  [\bomega_{i-1, \xi(i)}]\ .
\end{eqnarray*} It is now immediate that the assignment \begin{gather*}
x_i\to [\bomega_{i, \xi(i+1)}],\ \ \ f_i\to\bof_i,\ \ \ x[\alpha_i]\to [\bomega_{i,\xi(i+1)\pm 2}]^{\delta_{i,i_\diamond}}[\bomega(i,i+1)]^{1-\delta_{i,i_\diamond}}
\end{gather*}
defines a homomorphism 
of rings
$\iota: \cal A(\bx,Q_\xi)\to \cal K_0(\cal F_\xi)$.

\subsection{The elements $\iota(x[\alpha])$, $\alpha\in\Phi_{\ge -1}$}\label{iotaroot}
The formulae given in \eqref{iota1} and \eqref{iota2} can be rewritten as follows:
\begin{gather}\label{iota}\iota(x[\alpha_{i,j}])=[\bomega(i,j+1)]^{1-d_j}\ [\bomega_{i,\xi(i+1)\pm 2}]^{\delta_{i_\bullet, j_\bullet}d_j}\  [\bomega(i,j_\bullet+1)]^{d_j(1-\delta_{i_\bullet, j_\bullet})},\ \ j\ge i.
  \end{gather} 
 We shall prove this reformulation by induction on $j-i$. 
  Observe that  induction begins when $j=i$ by definition.  
  For the inductive step apply $\iota$ to both sides of the second equation in Proposition \ref{clusterind}. We will show that the right  hand side of this equation is the same as the right  hand side of the equation in Proposition \ref{clusterrep}(ii), (iii). Hence the left hand sides must match up. The inductive step is  immediate once we observe that $\cal K_0(\cal F_\xi)$ has no zero divisors.
  
To prove that the right hand sides are the same, suppose first that  $j_\bullet\le i$ (in particular $j_\bullet=i_\bullet$ or $j_\bullet=i$). Applying $\iota$ to both sides of the second equation in Proposition \ref{clusterind} gives
\begin{eqnarray*} &
[\bomega_{j,\xi(j+1)}]\iota(x[\alpha_{i,j}])&= \bof_j^{d_{j-1}}\iota(x[\alpha_{i,j-1}])[\bomega_{j+1,\xi(j+2)}]^{1-d_j}\\&&  +\ \bof_{j+1}^{1-d_j}\bof_i^{\delta_{i,j_\bullet}}[\bomega_{j+1,\xi(j+2)}]^{d_j}[\bomega_{i-1,\xi(i)}]^{1-\delta_{i,j_\bullet}}.
\end{eqnarray*}
The second term on the right hand side of  the preceding equation is equal to the the second term on the right hand side of the equation in Proposition \ref{clusterrep}(ii). To see that the first terms match up we  use the inductive hypothesis for $\iota(x[\alpha_{i,j-1}])$ and see that it suffices to prove that,
 \begin{eqnarray*}[\bomega_{i, \xi(i+1)\pm 2}]^{d_{j-1}}=\left([\bomega_{i,\xi(i+1)\pm 2}]^{\delta_{i_\bullet, (j-1)_\bullet}}[\bomega(i, (j-1)_\bullet+1)]^{1-\delta_{i_\bullet, (j-1)_\bullet}}\right)^{d_{j-1}}
  .\end{eqnarray*}
  If $d_{j-1}=0$, then the  preceding equality is obviously true. Since $$d_{j-1}=1\implies (j-1)=(j-1)_\diamond=j_\bullet=i\implies  i_\bullet=(j-1)_\bullet$$ and the equality  follows.

If   $i<j_\bullet$, then  the result follows if we prove that,\begin{gather*} 
\iota(x[\alpha_{i,j-1}])= [\bomega(i,j)]^{1-d_{j-1}}\left([\bomega_{i,z}]^{\delta_{k,i_\bullet}}[\bomega(i,k+1)]^{1-\delta_{k,i_\bullet}} \right)^{d_{j-1}}, \\ 
\iota(x[\alpha_{i,j_\bullet-1}])= [\bomega(i,j_\bullet)]^{1-d_{j_\bullet-1}}\left([\bomega_{i,z}]^{\delta_{i_\bullet,k_\bullet}}[\bomega(i,k_\bullet+1)]^{1-\delta_{i_\bullet,k_\bullet}}\right)^{d_{j_\bullet-1}}.
\end{gather*}  where we recall that $k=(j_\bullet)_\bullet$.  
If $d_{j-1}=0$ the first equality follows from the definition and the inductive hypothesis and  if $d_{j-1}=1$ then   $(j-1)=j_\bullet$ and so $(j-1)_\bullet =k$. The first equality again follows from the inductive hypothesis. The second equality is  deduced in the same way from the inductive hypothesis.

\subsection{} We prove now that $\iota$ is an isomorphism.  Let $\{\omega_1,\cdots,\omega_n\}$  which are dual to the simple roots of $A_n$ and $P^+$ be their  $\bz_+$--span. It is convenient to set $\omega_0=\omega_{n+1}=0$. Let $\le $ be the usual partial order on $P^+$ given by  $\mu\le \lambda$ iff $\lambda-\mu$ is in the $\bz_+$--span of $\{\alpha_1,\cdots,\alpha_n\}$. 

Define a morphism of monoids  $\wt:\cal P^+_\xi\to P^+$  by setting $\wt\bomega_{i,a}=\omega_i$.  Since $\cal F_\xi$ is a tensor category it is well--known that the following holds in $\cal K_0(\cal F_\xi)$; for  $\bomega=\bomega_{i_1,a_1}\cdots\bomega_{i_k,a_k}\in\cal P^+_\xi$: \begin{equation}\label{weylko}[\bomega_{i_1,a_1}]\cdots[\bomega_{i_k,a_k}]=[\bomega]+\sum_{\stackrel{\bpis\in\cal P^+_\xi}{\wt\bpis<\wt\bomegas}} r(\bomega,\bpi)[\bpi],\ \ {\rm{for\ some}}\ \ r(\bomega,\bpi)\in\bz_+.\end{equation} 
A straightforward induction on $\wt\bomega$ shows that 
 $\cal K_0(\cal F_\xi)$ is generated as a ring by the elements $[\bomega_{i,\xi(i)\pm 1}]$.
By Section \ref{iotaroot} we see that $\iota(\{x[-\alpha_i], x[\alpha_{i,i_\diamond}]\})=\{[\bomega_{i,\xi(i)+1}],[\bomega_{i,\xi(i)-1}]\}$ 
and hence it follows that
  $\iota$ is surjective.
 
We  prove that $\iota$ is injective. Set $$\wt_\ell x_i= \bomega_{i,\xi(i+1)}, \ \ \wt_\ell f_i=\bof_i, \  \wt_\ell x[\alpha_i] = \bpi, \ \ {\rm such \ that} \ \  \iota(x[\alpha_i]) = [\bpi].$$
Extend $\wt_\ell$ in the obvious way to the basis of $\cal A(\bx,\xi)$; if $\bom=
x_1^{p_1}\cdots x_n^{p_n}x[\alpha_1]^{m_1}\cdots x[\alpha_n]^{m_n}$  is a standard monomial in $\cal A(\bx,\xi)$ and   $\bof=f_1^{
r_1}\cdots f_n^{r_n}\in\bz[f_1^{\pm 1},\cdots ,f_n^{\pm 1}]$ then 
\begin{gather*}
\wt_\ell\bof
\bom= \prod_{i=1}^n\bof_i^{r_i}
\bomega_{i,\ \xi(i+1)}^{p_i}\left(\bomega_{i,\ \xi(i+1)+ 2}^{\delta_{\xi(i),\xi(i+1)+1}}\bomega_{i,\xi(i+1)-2}^{\delta_{\xi(i),\xi(i+1)-1}}
\bomega_{i+1,\ \xi(i+2)}
^{1-d_i}\right)^{m_i}.\end{gather*}

\begin{lem}\label{distinct}
Let $\bom,\bom'$ be standard monomials in $\cal A
(\bx, Q_\xi)$ and $\bof,\bof'$ be monomials in $\{f_i: i\in[1,n]\}$.
Then $$\wt_\ell\bof
\bom=\wt_\ell\bof'
\bom'\iff \bof=\bof'\ \ {\rm{and}}\ \ \bom=\bom'.$$
\end{lem}
\begin{pf} Write $$\bom=
x_1^{p_1}\cdots x_n^{p_n}x[\alpha_1]^{m_1}\cdots x[\alpha_n]^{m_n},\ \ \bof=f_1^{
r_1}\cdots f_n^{r_n},$$ and let $\bom',\bof'$ be defined similarly with $p_i$ replaced by $p_i'$ etc. If $p_1>0$ then $m_1=0$ and using the fact that $\cal P^+_\xi$ is a free abelian monoid
we  have $$\bof_1^{r_1}\bomega_{1, \xi(2)}
^{p_1}=
\bof_1^{r_1'}\bomega_{1,\ \xi(2)}^{p_
1'}\bomega_{1,\ \xi(2)+ 2
}^{m_1'\delta_{\xi(1), \xi(2)+1}}\bomega_{1,\ \xi(2)- 2
}^{m_1'\delta_{\xi(1),\xi(2)-1}}.$$ Since $\bof_1=\bomega_{1,\xi(1)+1}\bomega_{i,\xi(1)-1}$, we get
 $$r_1+p_1=r_1'+p_1',\ \ r_1=m_1'+r_1'.$$ If $m_
 1'\ne 0$ then $p_1'=0$ and we have $r_1>r_1'$ and 
 $r_1'>r_1$  which is absurd. Hence $m
_1'=0$ and so $r_1'=r_1$ and $p_1'=p_1$. Writing $\bom= x_1
^{p_1}\bom_1$ and $\bom'=x_1
^{p_1}\bom_1'$
we see  that $\bom_1$ and $\bom_1'$ are both standard monomials and  $$\wt_\ell f_2^{r_2}\cdots f_n^{r_n}\bom_1=\wt_\ell f_2^{r_2'}\cdots f_n^
{r_n'}
\bom_1'.$$
 An obvious iteration  of the preceding argument proves the Lemma.
\end{pf} 

Suppose that $$\iota\left(\sum_{r,s} c_{r,s}\bof(s) \bom_r\right)=0, $$
where   
 $\bom_r$ varies over 
 standard monomials in $\cal A(\bx,Q_\xi)$,
 and $\bof(s)$ varies over
  monomials in $f_i$, $i\in[1,n]$ and $c_{r,s}\in\bz$ with only finitely many being non--zero.
 Assume for a contradiction that $c_{r,s}\ne 0$ for some $r,s$ and let $\lambda$ be a maximal element (with respect to the partial order on $P^+$) of the set
 $\{\wt(\wt_\ell\bof(s)\bom_r): c_{r,s}
  \ne 0\}.$
  Using \eqref{weylko} we get  $$0=
  \sum_{\wt(\wt_\ell\bof(s) \bom_r)=\lambda}
  c_{r,s} [\wt_\ell\bof(s)\bom_r] + \sum_{ \wt\bomegas \ngtr \lambda  } n_\bomegas[\bomega],\ \ \ \  n_\bomegas\in \bz.
$$
Since the elements $[\bomega]$, $\bomega\in\cal P^+_\xi$ are linearly independent elements of $\cal K_0(\cal F_\xi)$ we get  $$  \sum_{\wt(\wt_\ell\bof(s) \bom_r)=\lambda}
c_{r,s} [\wt_\ell\bof(s)\bom_r] =0.$$ By Lemma \ref{distinct} the elements $
[\wt_\ell(\bof(s)\bom_r)]$ are all distinct and hence also linearly independent.  This forces  $c_{r,s}=0$ contradicting our assumption and  proves that $\iota$ is injective.

\subsection{The elements $\iota(x[\beta_1]x[\beta_2])$} 

We now prove the final assertion of the theorem. Write $[\bomega_s]=\iota([x[\beta_s])$, $s=1,2$ and let $\bomega=\bomega_1\bomega_2$. Assuming  that $[\bomega]\ne[\bomega_1][\bomega_2]$ we shall prove that $x[\alpha]x[\beta]$ is not a cluster monomial.
 By Proposition \ref{clustermon} we can write $[\bomega_1][\bomega_2]$ as the non-trivial sum of elements which are imaged  under $\iota$ of cluster monomials. Since cluster monomials are linearly independent and $\iota$ is an isomorphism we see that $x[\beta_1]x[\beta_2]$ is not a cluster monomial and the proof of the main theorem is complete.

  \section{Proof of Proposition \ref{clusterind} and a $q$-character formula.}\label{proofofclusterind}
  In this section we prove Proposition \ref{clusterind} which is a  recursive formula for a cluster variable. We also solve this recursions and give a closed formula for the cluster variable in terms of the initial cluster and the frozen variables. In view of Section \ref{iotaroot} this formula can also be viewed as giving the $q$--character of $[\bomega]$, $\bomega\in\bp\bor_\xi$ in terms of the local Weyl modules and Kirillov--Reshetikhin modules.
  
  \subsection{} \label{cluster}
  We briefly recall the definition (see \cite{FZ}) of a cluster algebra. Let $Q$ be a quiver with $(n+m)$-vertices labeled $\{1,\cdots, n, 1',\cdots, m'\}$ and assume that the set of edges has no  loops or $2$-cycles.   A mutation of $Q$ at a vertex $i$  is the quiver obtained by  performing   the following three operations.
  \begin{itemize}
  \item reverse all edges at $i$,
  \item given edges $j\to i\to k$ add a new  edge $j\to k$,
  \item remove any two cycles that may have been created.
  \end{itemize} 
  We shall assume that mutation is never allowed at the vertices labeled $\{1',\cdots,m'\}$; these  are called the frozen vertices.
  Suppose that $\bold x =\{x_1,\cdots, x_n, f_1,\cdots, f_m\}$ is an algebraically independent set and let $\mathbb Q (\bold x)$ be the field of rational functions in these variables. The set $\bold x$ is called the initial cluster and $(\bold x, Q)$ is called the initial seed. 
  
 Corresponding to a mutation of $Q$ at a vertex $i$ define a new cluster $\bold x'=\{x_1',\cdots, x_n',f_1,\cdots, f_m\}$ by 
   $$x_j'= x_j,\ \ j\neq i,\ \ x_i'x_i = \prod_{\substack{\exists\ j\to i\\ {\rm in} \ Q}}f_j\prod_{\substack{\exists\ j\to i\\ {\rm in} \ Q}}x_j + \prod_{\substack{\exists\ i\to k\\ {\rm in} \ Q}}f_k \prod_{\substack{\exists\ i\to k \\ {\rm in}\ Q}}x_k.$$ The new cluster again consists of algebraically independent elements and we have a new seed $(\bx', Q')$ where $Q'$ is the mutation of $Q$ at $i$. Iterating this process defines a collection of new clusters and new seeds. An element of a given cluster is called a cluster variable. A cluster monomial is a product of cluster variables all belonging to the same cluster.
   The associated cluster algebra is the $\bz$ subring (of the field of rational function $\mathbb Q(\bold x)$) generated by all the cluster variables.  

   \subsection{The quiver $ Q_\xi[i,j]$}\ 
Given $1\le i\le n-1$ set $Q_\xi= Q_\xi[i,j]$ if $j<i$ and let $Q_\xi[i,i]$ be obtained by mutating $Q_\xi$ at $i$. Assume that we have defined $Q_\xi[i,j-1]$ for $j>i$   let $Q_\xi[i,j]$ be the quiver defined   by mutating $Q_\xi[i,j-1]$ at $j$.   
\medskip

\noindent Proposition \ref{clusterind} is  a simple inspection  when $j=i$ and if $j>i$ then it is a consequence of the discussion in Section \ref{cluster}, the following Lemma  and an induction on $j-i$. 
  
\begin{lem}
  Suppose that $j> i
$  and that we have an arrow $(j-1)\to j$ in $Q_\xi$. In  $Q_\xi[i,j-1]$ we have the following edges at the vertex $j$:

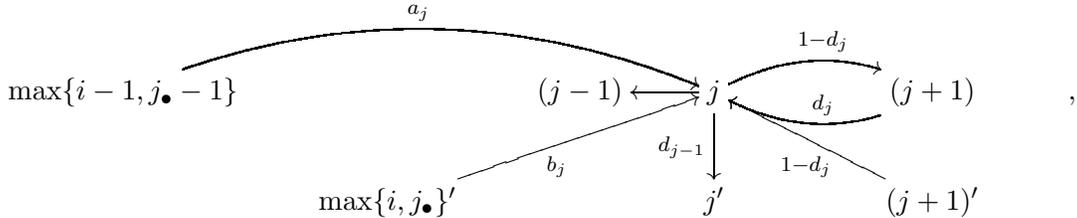
\vspace{5pt}

$$\xymatrix{
\max\{i-1,j_\bullet-1\}\ar@/^2pc/[rrr]^{a_j}  &  & (j-1) & j\ar[l]\ar[d]_{d_{j-1}} \ar@/^1pc/[rr]^{1-d_j} & & (j+1)\ar@/^1pc/[ll]_{d_j}\\
& \max \{i,j_{\bullet}\}'\ar[urr]_{b_j} & &j' & & (j+1)'\ar[llu]^{1-d_j} &
},$$
where $a_j = 1-\delta_{i,j_\bullet}$ and $b_j = \min\{1, (1-\delta_{j_\bullet,i_\bullet})d_{j_\bullet-1} + \delta_{j_\bullet,i}\}$.
\end{lem}
  \begin{proof}
  We proceed by induction on $j-i$. To see that induction begins when $j=i+1$ notice that \begin{gather*} d_i=1\implies i=(i+1)_\bullet\implies a_{i+1}=0, \ \ b_{i+1}=1,\\ d_i=0\implies i_\bullet= (i+1)_\bullet\implies \  a_{i+1}=1,\ \ b_{i+1}=0.\end{gather*} On the other hand in $Q_\xi[i,i]$ which is the 
mutation of  $Q_\xi$ at $i$ an inspection show that the edges at $i+1$ are given as follows: 
$$
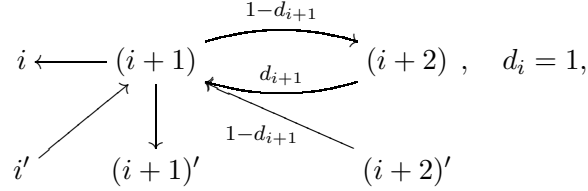
\xymatrix{ 
  i  & (i+1)\ar[l]\ar[d]\ar@/^1pc/[rr]^{1-d_{i+1}} & & (i+2)\ar@/^1pc/[ll]_{d_{i+1}}  \\
 i' \ar[ur] & (i+1)'  & & (i+2)' \ar[ull]^{1-d_{i+1}} 
 },\ \ \ d_i=1,
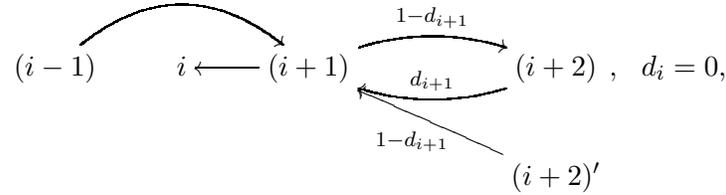
$$  
 
  $$\xymatrix{ 
(i-1)\ar@/^2pc/[rr] & i & (i+1) \ar[l]\ar@/^1pc/[rr]^{1-d_{i+1}} & & (i+2)\ar@/^1pc/[ll]_{d_{i+1}}  \\
 &  &  & & (i+2)' \ar[ull]^{1-d_{i+1}} 
 },\ \ d_i=0,$$
and it follows that induction begins.
For the inductive step we assume that the  result holds for the edges at  $j<n$  in  $Q_\xi[i,j-1]$ for and prove that it holds  for the node $j+1$ in $Q_\xi[i,j]$. 
\medskip

{\bf  Case 1}. If  $d_j=1$ then $j$ is a sink of $Q_\xi$ by assumption and so  we have an edge $(j+1)\to j$ in $Q_\xi$. Hence by the inductive hypothesis the edges at $j$ and $(j+1)$ in $Q_\xi[i,j-1]$  are
$$
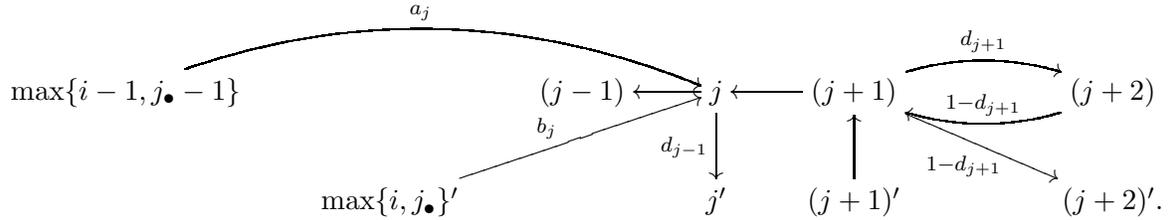
\xymatrix{
 \max\{i-1,j_\bullet-1 \}\ar@/^2pc/[rrr]^{a_j} &  & (j-1) & j\ar[l] \ar[d]_{d_{j-1}}& (j+1)\ar[l]\ar[rrd]_{1-d_{j+1}} \ar@/^1pc/[rr]^{d_{j+1}}& &(j+2)\ar@/^1pc/[ll]_{1-d_{j+1}} \\
& \max\{i,j_{\bullet}\}'\ar[urr]^{b_j}& &j' & (j+1)' \ar[u]& &(j+2)'.
}$$
Mutating at $j$ we see that the edges at $(j+1)$ are
$$
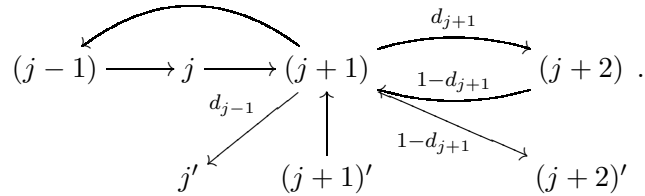
\xymatrix{
 (j-1)\ar[r] & j\ar[r]  & (j+1) \ar[ld]_{d_{j-1}}\ar@/_2pc/[ll]\ar[drr]_{1-d_{j+1}}\ar@/^1pc/[rr]^{d_{j+1}} & & (j+2)\ar@/^1pc/[ll]_{1-d_{j+1}}\\
&j'& (j+1)'\ar[u] & & (j+2)'
}.$$
The inductive 
step follows  since  $d_j=1\implies (j+1)_\bullet= j$ and so   $$\max\{i-1, (j+1)_\bullet-1\}= j-1,\ \ \max\{i, (j+1)_\bullet\}' = j',\ \  
\ \
a_{j+1} =  1=d_j,\ \  b_{j+1} = d_{j-1}.$$

{\bf Case 2.}
If $d_j=0$ or equivalently $j_\diamond\neq j$ then in $Q_\xi$ we have an edge $j\to j+1$. By the induction hypothesis, the edges at $j$ and $(j+1)$ in $Q_\xi[i,j-1]$ are 
$$
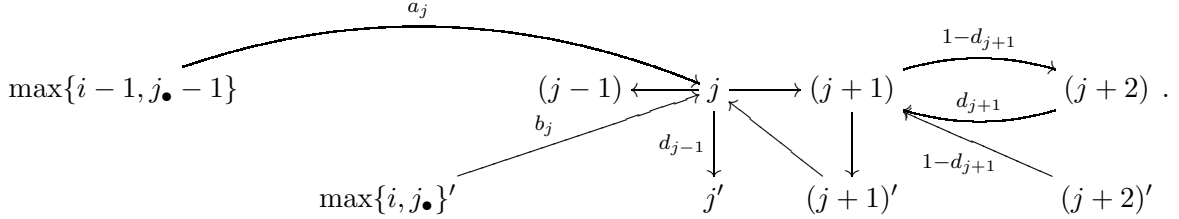
\xymatrix{
\max\{i-1,j_\bullet-1 \}\ar@/^2pc/[rrr]^{a_j} &  & (j-1) & j\ar[d]_{d_{j-1}} \ar[l]\ar[r]  & (j+1)\ar[d]\ar@/^1pc/[rr]^{1-d_{j+1}}& & (j+2) \ar@/^1pc/[ll]_{d_{j+1}}\\
& \max\{i,j_\bullet\}' \ar[urr]^{b_j}& &j' & (j+1)'\ar[ul] & & (j+2)'\ar[ull]^{1-d_{j+1}} 
}.\\ \\ $$

Mutating at $j$ we obtain 
$$
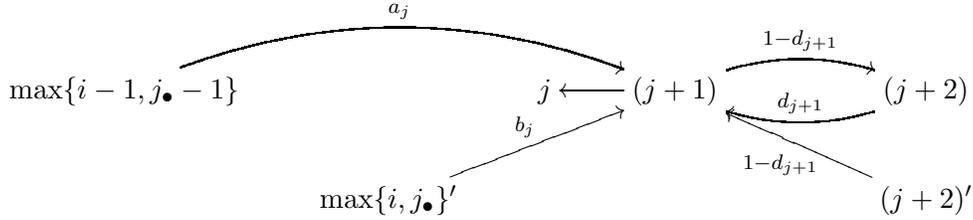
\xymatrix{
\max\{i-1,j_\bullet-1 \}\ar@/^2pc/[rrr]^{a_j} &   & j & (j+1)\ar[l]\ar@/^1pc/[rr]^{1-d_{j+1}} & &(j+2)\ar@/^1pc/[ll]_{d_{j+1}}\\
& \max\{i,j_\bullet\}'\ar[urr]^{b_j}&& & & (j+2)'\ar[ull]^{1-d_{j+1}}
}\\ \\ $$
The inductive step follows from the fact that $d_j=0\implies (j+1)_\bullet = j_\bullet<j$ and so
 $$\max\{i-1, (j+1)_\bullet-1\}= \max\{i-1, j_\bullet-1\},\ \ \max\{i, (j+1)_\bullet\}' = \max\{i, j_\bullet\}',$$ and $$a_{j+1} =  a_j,\ \ b_{j+1} =  b_j,\ \ d_{j} = 0.$$ The proof of the Lemma is complete.

  \end{proof}
  
  \subsection{The set $\Gamma_{i,j}$}  We continue to set $d_m=\delta_{m,m_\diamond}$ for $1\le m\le n$.
For $i,j\in[1,n]$ define sets $\Gamma_{i,j}$ as follows: $\Gamma_{i,j}=\{0\}$ if $  j<i $ and if $i\le j$ then $\Gamma_{i,j}$ is the subset of $\bz_+^{j-i+2}$ consisting of elements 
 $\bepsilon=(\epsilon_i,\cdots, \epsilon_{j+1})$ satisfying the following conditions: for $r,m\in[i,j]$ with $r\le m$ and  $\sigma_{r,m}(\bepsilon)=\epsilon_r+\cdots +\epsilon_m$,  we have
 \begin{gather} \label{ineq3}
\epsilon_{j+1}= 1+ (d_j-1)\sigma_{\max\{i,j_\bullet+1\}, j}(\bepsilon)\\
\label{ineq2}
\sigma_{\max\{i,j_\bullet+1\}, j}(\bepsilon)\le 1,\ \ \\
\label{ineq1a} \sigma_{i,i_\diamond}(\bepsilon)\le 1 \le  \sigma_{i,i_\diamond+1}(\bepsilon)\ \ {\rm{if}}\ \  i_\diamond\le j,\\
\label{ineq1}
\sigma_{m+1, (m+1)_\diamond}(\bepsilon)\le 1 \le \sigma_{m+1, (m+1)_\diamond+1}(\bepsilon),\ \ {\rm{if}}\ \  i_\diamond \le m=m_\diamond< j_\bullet
\end{gather}
   Clearly, $\epsilon_m\in\{0,1\}$ for $i\le m\le j+1$. For $i\le j$ let
 $$\Gamma_{i,j}^1=\{\bepsilon\in\Gamma_{i,j}: \sigma_{\max\{i,j_\bullet+1\}, j}(\bepsilon)=1\},\ \ \Gamma_{i,j}^0=\{\bepsilon\in\Gamma_{i,j}: \sigma_{\max\{i,j_\bullet+1\}, j}(\bepsilon)=0\}.$$ The condition in  \eqref{ineq2} shows that $$\Gamma_{i,j}=\Gamma_{i,j}^1\sqcup\Gamma_{i,j}^0.$$
 We shall use the following freely:
 \begin{equation}\label{trivob} d_{m-1}=0\iff (m-1)_\bullet=m_\bullet,\ \ d_{m-1}=1 \iff m_\bullet=m-1.\end{equation}
 \begin{lem} \label{gammaij}  For $j>i$ the assignments \begin{gather*}
(\epsilon_i,\cdots, \epsilon_j)\to  (\epsilon_i,\cdots, \epsilon_j, d_j),\ \ \  \\
(\epsilon_i,\cdots,\epsilon_{j_
\bullet}) \to (\epsilon_i+\delta_{i,j_\bullet},\cdots,\epsilon_{j_\bullet},0,\cdots,0, 1), \end{gather*} define bijections $\iota_{j-1}:\Gamma_{i,j-1}\to \Gamma_{i,j}^1$ and $\iota_{j_\bullet-1}:\Gamma_{i,j_\bullet-1}\to \Gamma_{i,j}^0$ respectively.
\end{lem} 
\begin{pf} For the first assertion of the Lemma we must  prove that $$\tilde\bepsilon=(\epsilon_i,\cdots,\epsilon_j)\in\Gamma_{i,j-1} \iff \bepsilon= (\epsilon_i,\cdots,\epsilon_j,d_j)\in\Gamma_{i,j}^1.$$
Clearly we have $ \sigma_{m,r}(\bepsilon)=  \sigma_{m,r}(\tilde\bepsilon)$ for all $ i\le m\le r\le j.$
Using \eqref{trivob} we see that
 $$(*)\ \ \epsilon_j=1+(d_{j-1}-1) \sigma_{\max\{i,(j-1)_\bullet+1\}, j-1}(\tilde\bepsilon)\iff \sigma_{\max\{i,j_\bullet+1\}, j}(\bepsilon)=1$$ It follows that $\tilde\bepsilon$ satisfies \eqref{ineq3} if $\bepsilon\in\Gamma_{i,j}^1$. It also proves that $\bepsilon$ satisfies \eqref{ineq3} and  \eqref{ineq2}
 if $\tilde\bepsilon\in\Gamma_{i,j-1}$. To see  that $\tilde\bepsilon$ satisfies \eqref{ineq2} if $\bepsilon\in\Gamma_{i,j}^1$ we note that this is clear if $(j-1)_\bullet=j_\bullet$ and if  $j_\bullet=j-1$ it follows from the fact that  $\bepsilon$ satisfies \eqref{ineq1} with $m=(j-1)_\bullet$. 
 
 It is obvious
 that $\tilde\bepsilon$ satisfies \eqref{ineq1a} (resp. \eqref{ineq1}) if $\bepsilon\in\Gamma_{i,j}^1$;  it is also obvious that $\bepsilon$ satisfies these inequalities if $\tilde\bepsilon\in\Gamma_{i,j-1}$  as long as $i_\diamond\le j-1$ (resp. $i_\diamond\le m< (j-1)_\bullet)$). 
  If  $i_\diamond=j$ then $d_j=1$ and $ j_\bullet=i_\bullet<i$. Using (*) and the fact that we have already proved that  $\bepsilon$ satisfies \eqref{ineq3} we get $$\sigma_{\max\{i,j_\bullet+1\},j}(\bepsilon)=1\le \max\sigma_{\max\{i,j_\bullet+1\},j+1}(\bepsilon)=2,$$ proving that \eqref{ineq1a} holds for $\bepsilon$.
 If $(j-1)_\bullet\le m=m_\diamond<j_\bullet$ then we must have  $  m=(j-1)_\bullet$, and $ j_\bullet=j-1=(m+1)_\diamond$.  It follows that $d_{j-1}=1$, $\epsilon_j=1$  and so we have
  $$\sigma_{(j-1)_\bullet+1, j-1}(\tilde\bepsilon)=\sigma_{(j-1)_\bullet+1, j-1}(\bepsilon)\le 1=\epsilon_j\le \sigma_{(j-1)_\bullet+1, j}(\bepsilon),$$    proving that $\bepsilon$ satisfies \eqref{ineq1}. 
The proof of the first assertion is complete.

We  prove the second assertion of the Lemma; note that if $\bepsilon\in\Gamma_{i,j}^0$ then we must have $\epsilon_m=0$ for $j_\bullet+1\le m\le j$ and hence by \eqref{ineq3} we also have $\epsilon_{j+1}=1$. Since $$j_\bullet\le i\implies \Gamma_{i,j_\bullet-1}=\{0\} \ \ {\rm{and}}\ \ \Gamma_{i,j}^0= \{(\delta_{i,j_\bullet},\cdots, 0, 1)\}$$  the result is trivially true in this case. We assume from now on that $j_\bullet> i$ (in particular $j_\bullet\ge i_\diamond$) and
let  $$\tilde\bepsilon = (\epsilon_i, \cdots, \epsilon_{j_\bullet})\ \qquad  \bepsilon = (\epsilon_i, \cdots, \epsilon_{j_\bullet}, 0, \cdots, 0,1).$$
Suppose  that $\tilde\bepsilon\in\Gamma_{i,j_\bullet-1}$.  It is obvious that $\bepsilon$ satisfies \eqref{ineq3} and \eqref{ineq2} and \eqref{ineq1a} and  for $i_\diamond\leq m < (j_\bullet-1)_\bullet$ that $\bepsilon$ satisfies \eqref{ineq1}. If  $ (j_\bullet-1)_\bullet\le m = m_\diamond\le j_\bullet-1$ then either $ m = (j_\bullet-1)_\bullet$ or $  m=j_\bullet-1$. In the first case  the first inequality in \eqref{ineq1} for $\bepsilon$ is just \eqref{ineq2} for $\tilde\bepsilon$ while the second inequality follows from \eqref{ineq3} for $\tilde\bepsilon$.
If $m=m_\diamond =j_\bullet-1$, then \eqref{ineq3} forces $\epsilon_{j_\bullet} =1$ and hence we have $\epsilon_{j_\bullet}\leq 1 \leq \epsilon_{j_\bullet} + \epsilon_{j_\bullet+1}$. This proves that \eqref{ineq1} holds for $\bepsilon$ and so $\bepsilon\in\Gamma_{i,j}^0$. 

Next we assume that  $\bepsilon\in\Gamma_{i,j}^0$ and prove that $\tilde\bepsilon\in\Gamma_{i,j_\bullet-1}$.
 To prove that \eqref{ineq3} holds for $\tilde\bepsilon$ it suffices to observe that 
if $j_\bullet=i_\diamond$ (resp. $(j_\bullet)_\bullet\ge i_\diamond$) then  \eqref{ineq1a} (resp. \eqref{ineq1})  for $\bepsilon$ gives 
$$ \sigma_{\max\{i,(j_\bullet)_\bullet+1\}, j_\bullet}(\tilde\bepsilon)=
\sigma_{\max\{i,(j_\bullet)_\bullet+1\}, j_\bullet}(\bepsilon)=1.$$
If $d_{j_\bullet-1} = 1$ then  $(j_\bullet)_\bullet = j_\bullet -1$ and so the preceding equality is $\epsilon_{j_\bullet}=1$ as needed. If $d_{j_\bullet-1} =0$ then $(j_\bullet-1)_\bullet =(j_\bullet)_\bullet$ and again the preceding equality is a reformulation of \eqref{ineq3} for $\tilde\bepsilon$.  
 The fact that $\tilde\bepsilon$ satisfies 
\eqref{ineq2}  follows by using  \eqref{ineq1a} for $\bepsilon$ if $( j_\bullet)_\bullet<i_\diamond$ and using  \eqref{ineq1} for $\bepsilon$ otherwise. It is clear that  \eqref{ineq1a} and \eqref{ineq1} hold for  $\tilde\bepsilon$ since they are the same as the corresponding ones for $\bepsilon$ and the proof of the Lemma is complete.
\end{pf}

\subsection{ The sets $\Gamma_{i,j}'$} For $i\leq j$ define a map $$p_{ij}: \Gamma_{i,j}\to \bz^{ (j-i+2)},\ \ 
p_{i,j}(\epsilon_i,\cdots,\epsilon_{j+1})=(\epsilon_i',\cdots,\epsilon_{j+1}'),$$
as follows: 
\begin{itemize}
\item $\epsilon_{j+1}'=  (1-d_j)\sigma_{\max\{i,j_\bullet+1\},j}(\bepsilon)+ d_j(1-\sigma_{\max\{i,j_\bullet+1\},j}(\bepsilon)),$\\
\item  if $i_\bullet =m_\bullet$ or  $\sigma_{\max\{i, (m_\bullet)_\bullet+1\},m_\bullet}=1$ 
  then, 
$$\epsilon_m'=\begin{cases}(d_m-1)
\epsilon_{m+1} - d_m,\ \ \ \  \sigma_{\max\{i, m_\bullet+1\},m}(\bepsilon)=0,\\  
d_m -\left(\epsilon_{m}+\epsilon_{m+1}\right) ,   \ \ \  \sigma_{\max\{i, m_\bullet+1\},m}(\bepsilon)=1,
\end{cases}$$
\item if $m_\bullet\ge i$ and $\sigma_{\max\{i, (m_\bullet)_\bullet+1,m_\bullet\}}(\bepsilon)=0$ then
$\epsilon_m'= d_m(1-\epsilon_{m+1})$.
\end{itemize}
\medskip
It is easily seen that $\epsilon_m'\in\{-1,0,1\}$ for $i\le m\le j$. For $i\le j$ let  $\Gamma_{i,j}'$ be the image of $p_{ij}$ and set $\Gamma_{i,j}'=\{0\}$ if $i>j$. 
\begin{lem}\label{gammaij'} Let $1\le i\le j\le n$.
\begin{enumerit}
\item[(i)] If $\tilde\bepsilon=(\epsilon_i,\cdots,\epsilon_j)\in\Gamma_{i,j-1}$ then $$p_{ij-1}(\tilde\bepsilon)=(\epsilon_i',\cdots,\epsilon_j')\implies p_{ij{}}(\iota_{j-1}(\tilde\bepsilon))=(\epsilon_i',\cdots,\epsilon_{j-1}', -1+\epsilon_j', 1-d_j).$$
\item[(ii)] If $\tilde\bepsilon=(\epsilon_i,\cdots,\epsilon_{j_\bullet})\in\Gamma_{i,j_\bullet-1}$ then
$$p_{ij_\bullet-1}(\tilde\bepsilon)=(\epsilon_i',\cdots,\epsilon_{j_\bullet}')\implies p_{ij{}}(\iota_{j_\bullet-1}(\tilde\bepsilon))=(\epsilon_i',\cdots,\epsilon_{j_\bullet}', 0\cdots, 0, -1, d_j).$$\end{enumerit}
\end{lem}
\begin{pf}
Let $\bepsilon = (\epsilon_i,\cdots, \epsilon_j, d_j)=  \iota_{i,j-1}(\tilde\bepsilon)$ and let  $p_{ij}(\bepsilon)= (\epsilon_i'',\cdots,\epsilon_{j+1}'') $.  Since $$\sigma_{m,r}(\tilde\bepsilon)=\sigma_{m,r}(\bepsilon),\ \ m\le r\le j$$  it is clear from the definition that $\epsilon_m'=\epsilon_m''$ if $m\le j-1$.  

By Lemma \ref{gammaij} we have $\bepsilon\in\Gamma_{i,j}^1$ and hence $\sigma_{\max\{i,j_\bullet+1\}, j}(\bepsilon)=1$. It is immediate from the definition of $p_{ij}$ that $\epsilon_{j+1}''=1-d_j$. We now prove that $\epsilon_{j}''= -1+\epsilon_j'$; using the definition of $\epsilon_j'$ this is equivalent to proving\begin{eqnarray}\notag&
\epsilon_j''&=-1+ (1-d_{j-1})\sigma_{\max\{i, (j-1)_\bullet+1\},j-1}(\tilde\bepsilon)+d_{j-1}(1-\sigma_{\max\{i,(j-1)_\bullet+1,\},j-1}(\tilde\bepsilon)),\\ \label{rewrite}
&&=-1 + (1-d_{j-1})\sigma_{\max\{i,  j_\bullet+1\},j-1}(\bepsilon) +d_{j-1}(1-\sigma_{\max\{i,(j_\bullet)_\bullet+1\},j_\bullet}(\bepsilon)).\end{eqnarray}
If $j_\bullet\ge i$ and $\sigma_{ \max\{i,(j_\bullet)_\bullet+1\}, j_\bullet}(\bepsilon)=0$ then by \eqref{ineq1} we have $\epsilon_{j_\bullet+1}=1$ and so $\sigma_{j_\bullet+1, j-1}(\bepsilon)=1$. This means that the right hand side of \eqref{rewrite} is zero. Since  by   definition $\epsilon_j''=d_j(1-d_j)=0$ the result is proved in this case.
\medskip

\noindent If $i_\bullet=j_\bullet$ then $d_{j-1}=0$ and since  $\sigma_{\max\{i, j_\bullet+1\},j}(\bepsilon)=1$ it follows that the right hand side of \eqref{rewrite} is $-\epsilon_j'$ which is precisely the value of  $\epsilon_j''$ in this case.
\medskip

\noindent Suppose that $\sigma_{\max\{i, (j_\bullet)_\bullet+1\},j_\bullet}(\bepsilon)=1$ and that $j_\bullet\geq i$. This means that the second term on the right hand side of \eqref{rewrite} is zero. Since $\sigma_{\max\{i, j_\bullet+1\},j}(\bepsilon)=1$ by definition we have $\epsilon_j''=-\epsilon_j$. Recalling that $ \epsilon_j =1+(d_{j-1}-1)\sigma_{\max\{i, (j-1)_\bullet+1\},j-1}$ we see that the right hand side of \eqref{rewrite}  is also $-\epsilon_j'$.
The proof of part (i) is now complete.

\medskip

We prove part (ii).  Let $$\bepsilon=  \iota_{i,j_\bullet-1}(\tilde\bepsilon)\in\Gamma_{i,j}^0,\ \ p_{ij}(\bepsilon)= (\epsilon_i'',\cdots,\epsilon_{j+1}'').$$ Since $\sigma_{m,r}(\tilde\bepsilon)=\sigma_{m,r}(\bepsilon)$ for all $ m\le r\le j_\bullet-1$  it is clear from the definition that $\epsilon_m'=\epsilon_m''$ if $m\le j_\bullet-1$. Since $d_m=0$ if $j_\bullet+1\le m\le j-1$ a simple inspection also shows that $$  \epsilon_m''=0,\ \ j_\bullet+1\le m\le j-1,\ \ \epsilon_{j+1}''= d_j.$$ 
It remains to prove that $\epsilon_{j_\bullet}''=\epsilon_{j_\bullet}'$ and that $\epsilon_j''=-1$.

\medskip
\noindent If $j_\bullet=i_\bullet$ then $\tilde\bepsilon=\{0\}$,  $\Gamma_{i,j_\bullet-1}=\{0\}$  and $\bepsilon=(0,\cdots, 0,1)$.  By definition  $p_{i,j}(\bepsilon) = (0,\cdots, 0,-1,d_j)$ and we are done in this case.  If $i=j_\bullet$ then $\Gamma_{i,j_\bullet-1}=\{0\}$ and $\bepsilon = (1,0,\cdots, 1)$.  and one  checks easily  that $\epsilon_i'= 0$. On the other hand, by definition $\epsilon_i'' = d_i - (\delta_{i,j_\bullet} + \epsilon_{i+1}) = 0$. The fact that $\epsilon_j'' = -1$ is a straightforward checking from the definition. 

\medskip
 
\noindent Suppose that $j_\bullet> i$.  Since 
 $\epsilon_{j_\bullet+1} =0$ we have $\sigma_{\max\{i,j_\bullet+1\}, j}(\bepsilon) =0$ and by using \eqref{ineq1} that  $\sigma_{\max\{i,(j_\bullet)_\bullet+1\}, j_\bullet}(\bepsilon) =1$. Since $\epsilon_{j+1}=1$ it  follows  by definition that   $\epsilon_j''= -1$ as needed.
 \medskip
 
 \noindent 
Finally to show  $\epsilon_{j_\bullet}''=\epsilon_{j_\bullet}'$,  we  write $m=j_\bullet$ and see that we must prove  \begin{eqnarray}\notag &\epsilon_{m}''&= (1-d_{m-1})\sigma_{\max\{i,(m-1)_\bullet+1\},m-1}(\tilde\bepsilon) + d_{m-1}(1- \sigma_{\max\{i,(m-1)_\bullet+1\},m-1}(\tilde\bepsilon))\\
&&= (1-d_{m-1})\sigma_{\max\{i,m_\bullet+1\},m-1}(\bepsilon) + d_{m-1}(1- \sigma_{\max\{i,(m_\bullet)_\bullet+1\}, m_\bullet}(\bepsilon)),\notag\\  &&=(1-d_{m-1})(1-\epsilon_m)+d_{m-1}(1- \sigma_{\max\{i,(m_\bullet)_\bullet+1\}, m_\bullet}(\bepsilon)).\label{e:ibul}
\end{eqnarray}
 If  $\sigma_{\max\{i,(m_\bullet)_\bullet+1\}, m_\bullet}(\bepsilon))=1$ then   $\epsilon_m''=1-\epsilon_m$ by definition. By \eqref{ineq1} we have  $\epsilon_m=1$ if $d_{m-1}=1$ and hence  $(1-\epsilon_m)=(1-d_{m-1})(1-\epsilon_m)$ and \eqref{e:ibul} is proved.
   If  $\sigma_{\max\{i,(m_\bullet)_\bullet+1\}, m_\bullet}(\bepsilon)=0$ then by definition $\epsilon_m''=1$. Hence we must prove that $$1=(1-d_{m-1})(1-\epsilon_m)+d_{m-1}.$$ If $d_{m-1}=1$ this is clear from the preceding computation. If $d_{m-1}=0$ then $m_\bullet+1<m$ and \eqref{ineq1} forces $\epsilon_{m_\bullet+1}=1$; in particular it follows that $\epsilon_m=0$ and \eqref{e:ibul}  and  is completely proved.
\end{pf}

\subsection{}

\begin{prop}\label{closedform} For  $1\le i\le j\le n$ we have  $$x[\alpha_{i,j}] =\sum_{\bepsilons\in \Gamma_{i,j}}f_{i,j}^\bepsilons m_{i,j}^\bepsilons,$$
 where $$m_{i,j}^{\bepsilons} = x_{i-1}^{1-\epsilon_i}x_i^{\epsilon_i'}\cdots x_{j}^{\epsilon_j'}x_{j+1}^{\epsilon_{j+1}'},\ \ f_{i,j}^\bepsilons= f_i^{\epsilon_i}\cdots f_j^{\epsilon_j}f_{j+1}^{(1-d_j)\epsilon_{j+1}},$$ 
with $\bepsilon = (\epsilon_i,\cdots, \epsilon_{j+1})$ and $p_{i,j}(\bepsilon) = (\epsilon_i', \cdots, \epsilon_{j+1}')$.
\end{prop}
\begin{pf}
The proof of the proposition proceeds by an induction on $j-i$. To see that induction begins recall from Proposition \ref{clusterind} that
$$x[\alpha_i]=\delta_{i,i_\diamond}\left( f_ix_i^{-1} + x_{i-1}x_{i+ 1}x_i^{-1}\right)+ (1-\delta_{i,i_\diamond})\left(x_{i}^{-1}\left(f_ix_{i+1} + x_{i-1}f_{i+1}\right)\right).$$
Since $$i=i_\diamond \implies \Gamma_{i,i}=\{(1,1), (0,1)\},\ \ \Gamma_{i,i}'=\{(-1,0),(-1,1)\}$$ and $$i\ne i_\diamond\implies \Gamma_{i,i}=\{(0,1), (1,0)\},\ \ \Gamma_{i,i}'=\{(-1,0), (-1,1)\}$$
we see that induction begins.
For the inductive step Proposition \ref{clusterind} asserts that
\begin{eqnarray*}\label{prop14}&x_jx[\alpha_{i,j}]& = f_j^{d_{j-1}}x[\alpha_{i,j-1}]x_{j+1}^{1-d_j} +\\&& f_{j+1}^{1-d_j}x_{j+1}^{d_j} \left((\delta_{i_\bullet,j_\bullet}+\delta_{i,j_\bullet})f_i^{\delta_{i,j_\bullet}}x_{i-1}^{1-\delta_{i,j_\bullet}}+(1-\delta_{i_\bullet,j_\bullet}-\delta_{i,j_\bullet}) f_{j_\bullet}^{d_{j_\bullet-1}} x[\alpha_{i,j_\bullet-1}]\right).
\end{eqnarray*}
Let $\tilde\bepsilon = (\epsilon_i,\cdots, \epsilon_j)\in \Gamma_{i,j-1}$. By Lemmas \ref{gammaij} and \ref{gammaij'} we have 
$$m_{i,j}^{\bepsilons} f_{i,j}^{\bepsilons} = m_{i,j-1}^{\tilde\bepsilons} f_{i,j}^{\tilde\bepsilons} f_{j}^{d_{j-1}}x_{j+1}^{1-d_j}x_j^{-1}, \quad \bepsilon = \iota_{i,j-1}(\tilde\bepsilon)$$
once we notice that $(1-d_{j-1})\epsilon_j +d_{j-1} = \epsilon_j$.  Hence using the inductive hypothesis we get $$f_j^{d_{j-1}}x[\alpha_{i,j-1}]x_{j+1}^{1-d_j}x_j^{-1}=\sum_{\bepsilon\in\Gamma_{i,j}^1}f_{i,j}^\bepsilons m_{i,j}^\bepsilons.$$
Similarly let $\tilde\bepsilon = (\epsilon_i,\cdots, \epsilon_{\max\{i,j_\bullet+1\}})\in \Gamma_{i,j_\bullet-1}$ and $\bepsilon = \iota_{i,j_\bullet-1}(\tilde\bepsilon)$.  Then $\epsilon_{j+1} = 1-d_j$ by definition and   $1-\epsilon_i = 1- \delta_{i,j_\bullet}$  if $j_\bullet\leq i$ and we get
$$m_{i,j}^{\bepsilons} f_{i,j}^{\bepsilons}= \begin{cases} f_{j+1}^{1-d_j}f_i^{\delta_{i,j_\bullet}}x_{i-1}^{1-\delta_{i,j_\bullet}}x_{j}^{-1}x_{j+1}^{d_j}, \quad \quad j_\bullet\leq i\\f_{i,j_\bullet-1}^{\tilde\bepsilons}f_{j+1}^{1-d_j}f_{j_\bullet}^{d_\bullet-1} m_{i,j_\bullet-1}^{\tilde\bepsilons}x_j^{-1}x_{j+1}^{d_j}, \quad\quad i<j_\bullet.
\end{cases}$$
The inductive step follows from the inductive hypothesis and  the fact that $\Gamma_{i,j} = \Gamma_{i,j}^0\sqcup \Gamma_{i,j}^1$.
\end{pf}
  
\section{Irreducible tensor products.}\label{irrtp} In
this section we  give a necessary condition (see Section \ref{irred})  for the equality $[\bpi_1][\bpi_2]=[\bpi_1\bpi_2]$ to hold when $\bpi_1,\bpi_2\in\bp\bor_\xi$.  We shall see in later sections that the conditions are sufficient as well. 
We shall often  need to  work in the monoidal  category  $\cal F_\xi$ rather than its Grothendieck ring;  by   abuse of notation we shall use the symbol $[\bomega]$ to also denote an  irreducible module in  $\cal F_\xi$ with label $\bomega$. To emphasize that we are working in the category  we shall write $[\bomega]\otimes [\bomega']$ for the tensor product of the corresponding objects.

\subsection{}\label{factsonweyls}  We collect some well--known facts on the category $\cal F_\xi$. An object of $\cal F_\xi$ is said to be {\em $\ell$-highest weight with highest weight $\bomega$} if it has $[\bomega]$ as its  unique irreducible quotient. Clearly any quotient of an $\ell$--highest weight module is also  $\ell$--highest weight with the same irreducible quotient.  Given $\bomega_1,\bomega_2\in\cal P_\xi^+$ the module $[\bomega_1\bomega_2]$ occurs in the Jordan--Holder series of $[\bomega_1]\otimes [\bomega_2]$  with multiplicity one. In particular if  $[\bomega_1]\otimes [\bomega_2]$ is an $\ell$--highest weight module then $[\bomega_1\bomega_2]$ is its unique irreducible quotient and hence  $[\bomega_1]\otimes [\bomega_2]$ is irreducible iff $[\bomega_1]\otimes [\bomega_2]\cong[\bomega_2]\otimes[\bomega_1]\cong[\bomega_1\bomega_2]$.

The following results from  \cite{H5}, \cite{H6} play an important role in this section.
\begin{thm}\label{pairs}  Let $\bomega_s\in\cal P^+_\xi$ for $1\le s\le r$. Then $[\bomega_1]\otimes\cdots\otimes [\bomega_s]$ is $\ell$--highest weight if   every pair $[\bomega_s]\otimes [\bomega_p]$ with $1\le s< p\le r$ 
is $\ell$--highest weight. Moreover if $[\bomega_s]\otimes [\bomega_p]\cong[\bomega_s\bomega_p]$ for all $1\le s<p\le r$ then $$[\bomega_1]\otimes\cdots\otimes[\bomega_r]\cong[\bomega_1\cdots\bomega_r].$$\hfill\qedsymbol
\end{thm}

\subsection{} 
Given $\bomega_{i,a},\bomega_{j,b}\in\cal P_\xi^+$ it was shown in \cite{CBraid} that the module $[\bomega_{i,a}]\otimes[\bomega_{j,b}]$ is $\ell$--highest weight (resp. irreducible) if 
$$(b-a)\notin\{2p+2-i-j: \max\{i,j\}<p+1\le\min\{n+1,i+j\},$$
$${\rm{(resp.}} \ \  \pm(b-a)\notin\{2p+2-i-j: \max\{i,j\}<p+1\le\min\{n+1,i+j\}).$$ The next proposition  is a simple calculation using the preceding criterion and the fact that $|\xi(j)-\xi(i)|\le |j-i|$ for all $i,j\in[1,n]$.
\begin{prop}\label{braid} Let  $\bomega_{i,a},\bomega_{j,b}\in\cal P^+_\xi$. Then $[\bomega_{i,a}]\otimes[\bomega_{j,b}]$ (resp. $[\bomega_{j,b}]\otimes[\bomega_{i,a}]$) is $\ell$--highest weight if   $a=\xi(i)+1$ (resp. $a=\xi(i)-1$). Moreover,
$$\bomega_{i,a}\bomega_{j,b}\notin\bp\bor_\xi\cup\{\bof_i\}\implies[\bomega_{i,a}]\otimes[\bomega_{j,b}]\cong[\bomega_{i,a}\bomega_{j,b}]. $$ \hfill\qedsymbol
\end{prop}

\subsection{}\label{admis} Let $\xi^*$ be the height function defined by $\xi^*(i)=\xi(n+1-i)$. The assignment $$\bomega_{i,\xi(i)\pm 1}\to\bomega_{n+1-i, \xi^*(n+1-i)\pm 1}$$ extends to  an   isomorphism  $\cal P^+_\xi\cong \cal P^+_{\xi^*}$,  and if  $\bomega=\bomega_{i_1,a_1}\cdots\bomega_{i_k,a_k}\in\cal P_\xi^+$ we set   $$\bomega^*= \bomega_{n+1-i_1,a_1}\cdots\bomega_{n+1-i_k,a_k}\in\cal P^+_{\xi^*}.$$  It  was  proved  in \cite{CPminaff} that  if $[\bomega_1]
\otimes [\bomega_2]$  and  $[\bomega_2^*]\otimes [\bomega_1^*]$ are both $\ell$--highest  weight then they are both irreducible with  the converse being   trivially true. 
\medskip

Say that an ordered  triple of elements $(\bomega_1,\bomega_2,\bomega_3)$ from $\cal P^+_\xi$ is $\xi$-admissible if: \begin{itemize}
\item $[\bomega_s]\otimes [\bomega_3]$ is  irreducible for $s=1,2$,
\item either $[\bomega_1]\otimes[\bomega_2]$
  or $[\bomega_2]\otimes[\bomega_1]$
is  $\ell$--highest weight 
\end{itemize}
\begin{lem}
If $(\bomega_1,\bomega_2,\bomega_3)$
is  $\xi$-admissible  and  $(\bomega_1^*,\bomega_2^*,\bomega_3^*)$ is $\xi^*$-admissible  then $[\bomega_1\bomega_2]\otimes[\bomega_3]\cong[\bomega_1\bomega_2\bomega_3].$
\end{lem} 
\begin{pf} Suppose that  $(\bomega_1,\bomega_2,\bomega_3)$
is  $\xi$-admissible and that $[\bomega_1]\otimes[\bomega_2]$ is $\ell$--highest weight. Then Theorem \ref{pairs} shows that the modules $[\bomega_1]\otimes[\bomega_2]\otimes[\bomega_3]$ and $[\bomega_3]\otimes[\bomega_1]\otimes[\bomega_2]$ are $\ell$--highest weight. Hence the corresponding quotients $[\bomega_1\bomega_2]\otimes[\bomega_3]$ and $[\bomega_3]\otimes[\bomega_1\bomega_2]$ are $\ell$--highest weight.  Working with $\xi^*$ we see similarly that $[\bomega_1^*\bomega_2^*]\otimes[\bomega_3^*]$ and $[\bomega_3^*]\otimes[\bomega_1^*\bomega_2^*]$ are $\ell$--highest weight. The irreducibility of the four quotient modules follows from the discussion preceding the statement of the Lemma.

\end{pf}

 \subsection{} Recall the map $\wt:\cal P^+_\xi\to P^+$ given by extending $\wt\bomega_{i,a}=\omega_i$ to a morphism of monoids; for $\bpi\in\cal P_\xi^+$ with  $\wt\bpi=\sum_{i=1}^n r_i\omega_i$  set  \begin{gather*} \Ht\bpi=\sum_{i=1}^nr_i,\ \  \min\bpi=\min\{i\in I: r_i\neq 0\},\ \  \max\bpi=\max\{i\in I: r_i\neq 0\}.\end{gather*} If $\bpi\in \bp\bor_\xi$ and  $b\in\{\xi(j)+1,\xi(j)-1\}$, $j\in[1,n]$ are such that $\bomega_{j,b}^{-1}\bpi\in\bp\bor_\xi$ then $j\in\{\min\bpi,\max\bpi\}$. \begin{prop} \label{split} Let $\xi$ be an arbitrary height function and let $\bomega,\bomega'\in\bp\bor_\xi$.
\begin{enumerit}
\item[(i)] Suppose that  $\bomega\bomega'\in\bp\bor_\xi$ and  write $\bomega=\bomega_1\bomega_{i,a}$ with $\bomega_1\in\bp\bor_\xi$ and $\bomega_{i,a}\bomega'\in\bp\bor_\xi$.
If $a=\xi(i)+1$  then $[\bomega]\otimes[\bomega']$ is $\ell$--highest weight and otherwise   $[\bomega']\otimes[\bomega]$ is $\ell$--highest weight.
\item[(ii)] If $i=\max\bomega<j=\min\bomega'$ and  $|\xi(i)-\xi(j)|\ne j-i$, then  the module $[\bomega]\otimes[\bomega']$ is 
irreducible.
\end{enumerit}
\end{prop}
\begin{pf} The proof of both parts is by an induction on $p=\Ht\bomega+\Ht\bomega'$ with  Proposition \ref{braid} showing  that induction begins when $p=2$.  For the inductive step assume that we have proved both parts for $p'<p$ and also assume  without loss of generality that $\Ht\bomega'\ge 2$. 
 
For  part (i)  write $$\bomega'=\bomega_{j,b}\bomega'',\ \ {\rm{with}}\ \  \bomega''\in\bp\bor_\xi,\ \ \bomega\bomega_{j,b}\in\bp\bor_\xi$$ and observe that
if $k\in\{\min\bomega'',\max\bomega''\}$ then $|\xi(i)-\xi(k)|\ne |k-i|$.   The inductive hypothesis  applies to the pairs $(\bomega_{j,b},\bomega'')$ and to $(\bomega,\bomega_{j,b})$; hence either both of  $[\bomega_{j,b}]\otimes [\bomega'']$ and  $[\bomega_{j,b}]\otimes [\bomega]$ or both of  $[\bomega'']\otimes [\bomega_{j,b}]$ and $[\bomega]\otimes [\bomega_{j,b}]$ are $\ell$--highest weight. 
The inductive hypothesis from part (ii)   applies to the pair $(\bomega, \bomega'')$ and so $[\bomega]\otimes[\bomega'']$ is irreducible. It follows from Theorem \ref{pairs} that the module $[\bomega_{j,b}]\otimes[\bomega'']\otimes[\bomega]$ is $\ell$--highest weight (or $[\bomega]\otimes [\bomega'']\otimes[\bomega_{j,b}]$ is $\ell$--highest weight). Hence the quotient $[\bomega']\otimes[\bomega]$ (or $[\bomega]\otimes[\bomega'])$ is $\ell$--highest weight and the inductive step for (i) is proved.

For part (ii) we continue to write $\bomega'=\bomega_{j,b}\bomega''$ and observe that the inductive hypothesis applies to the pairs $(\bomega, \bomega_{j,b})$ and $(\bomega, \bomega'')$ and gives that $[\bomega]\otimes[ \bomega_{j,b}]$ and $[\bomega]\otimes [\bomega'']$ are irreducible. Since  the inductive   step has been proved for  part (i) it  applies to the pair $(\bomega_{j,b},\bomega'')$ and so we see that $(\bomega'',\bomega_{j,b}, \bomega)$ is $\xi$-admissible. The conditions of the proposition obviously hold for $\xi$ iff they hold for $\xi^*$; hence it follows from Lemma \ref{admis} that $[\bomega]\otimes[\bomega']$ is irreducible. 

 \end{pf}

  \subsection{} The next proposition is essential to prove our main result. 

\begin{prop}\label{hl2result} Suppose that $1\le j_1\le j_2\le j_3\le j_4\le n$ are such that $|\xi(j_s)-\xi(j_{s+1})|= j_{s+1}-j_s$ for all $1\le s\le 3$.
\begin{enumerit}
\item[(i)] Let $j_1<j_2$ and $\bomega_{j_3,a}\in\bp\bor_\xi$. The module 
 $[\bomega(j_1,j_2)]\otimes[\bomega_{j_3,a}]$ is irreducible if  $$ \bomega(j_1,j_2)\bomega_{j_3,a}\notin\{\bomega(j_1,j_3),\ \ \bof_3\bomega_{j_1,\xi(j_1)\pm 1}\}.$$  An analogous statement holds if $j_2<j_3$ and $\bomega_{j_1,a}\in\bp\bor_\xi$.\\

\item[(ii)] Let $j_1<j_2<j_3$. The following modules are irreducible :\\ $\bullet$ $[\bomega(j_1,j_2)]\otimes [\bomega(j_1,j_2)]$,\\
 $\bullet$ $[\bomega(j_1,j_2)]\otimes[\bomega(j_1,j_3)]$   if $\xi(j_2-1)\ne \xi(j_2+1)$,\\
 $\bullet$ $[\bomega(j_1,j_2)]\otimes[\bomega (j_2,j_3)]$ if 
 $\xi(j_2-1)=\xi(j_2+1)$. \\ 

\item[(iii)] Assume  that $j_1<j_2<j_3<j_4$. Then the following are irreducible:\\
$\bullet$
$[\bomega(j_1,j_4)]\otimes[\bomega(j_2,j_3)]$  if  $|\xi(j_4)-\xi(j_1)|=j_4-j_1$,\\  $\bullet$  $[\bomega(j_1,j_2)]\otimes[\bomega(j_3,j_4)]$ if   $\bomega(j_1,j_2)\bomega(j_3,j_4)\ne \bomega(j_1,j_4).$ 
  \end{enumerit}
 \end{prop}
 \begin{pf}   Part (i) was proved in \cite{HL2} if $|\xi(j_3)-\xi(j_1)|=j_3-j_1$. If $|\xi(j_3)-\xi(j_1)|\ne j_3-j_1$writing $\bomega(j_1,j_2)=\bomega_{j_1,a_1}\bomega_{j_2,a_2}$,   our assumptions force,$$j_2\ne j_3, \ \    \xi(j_2-1)=\xi(j_2+1),\ \   \bomega_{j_2,a_2}\bomega_{j_3,a}\ne \bomega(j_2,j_3).$$ Proposition \ref{braid} now shows that  $(\bomega_{j_1,a_1},\bomega_{j_2,a_2},\bomega_{j_3,a})$ is a  $\xi$--admissible triple. It also proves that 
 $(\bomega_{n+1-j_1,a_1},\bomega_{n+1-j_2,a_2},\bomega_{n+1-j_3,a})$ is $\xi^*$-admissible and the hence Lemma \ref{admis} gives the result.
 The proof of the analogous statement for $\bomega_{j_1,a}$ is entirely similar.
 \medskip
 
 \noindent
 
 The first two assertions in part (ii) were proved in \cite{HL2}. Suppose that $j_2<j_3$ and $\xi(j_2-1)=\xi(j_2+1)$  and write $$\bomega(j_1,j_2)=\bomega_{j_1,a_1}\bomega_{j_2,a_2},\ \ \bomega(j_2, j_3)=\bomega_{j_2,a_2}\bomega_{j_3,a_3}.$$ Then $$a_1=\xi(j_1)\pm 1\iff a_2=\xi(j_2)\mp 1, \ \ a_3=\xi(j_3)\pm 1.$$ 
 Assuming that $a_1=\xi(j_1)+1$ we use  Proposition \ref{split} and Theorem \ref{pairs} and part (i) of this proposition to see that  $ [\bomega_{j_3,a_3}]\otimes [\bomega_{j_2,a_2}]\otimes[\bomega(j_1,j_2)]$ is $\ell$--highest weight and hence so is the quotient $[\bomega(j_2,j_3)]\otimes [\bomega(j_1,j_2)]$. Similarly working with $ [\bomega_{j_1,a_1}]\otimes [\bomega_{j_2,a_2}]\otimes[\bomega(j_2,j_3)]$ we see that $[\bomega(j_1,j_2)]\otimes [\bomega(j_2,j_3)]$ is $\ell$--highest weight. Repeating the argument with $\xi^*$ proves the irreducibility and proves the third assertion of part (ii).
 \medskip
 
 \noindent
 
The first assertion in (iii) was proved in \cite{HL2}.
 If $|\xi(j_4)-\xi(j_1)|\neq j_4-j_1$ then  either  $\xi(j_2-1)=\xi(j_2+1)\ \ {\rm{or}}\ \ \xi(j_3-1)=\xi(j_3+1)$.  Write 
 $$\bomega(j_1, j_2)=\bomega_{j_1,a_1}\bomega_{j_2,a_2},\qquad\bomega(j_3,j_4)=\bomega_{j_3,a_3}\bomega_{j_4,a_4},$$ and observe that  since $j_2<j_3$ and  $\bomega(j_1,j_2)\bomega(j_3,j_4)\ne \bomega(j_1,j_4)$ we get 
  \begin{gather*} (a_1, a_2)=(\xi(j_1)\mp 1, \xi(j_2)\pm 1)\iff (a_3,a_4)=(\xi(j_3)\pm 1,\xi(j_4)\mp  1). 
  \end{gather*} 
  If  $\xi(j_3-1)=\xi(j_3+1)$ then  $|\xi(j_4)-\xi(j_2)|\ne (j_4-j_2)$. Using parts (i) and (ii) of the current  proposition  and Proposition \ref{split} we  see that  $(\bomega_{j_3,a_3},\bomega_{j_4,a_4}, \bomega(j_1,j_2))$ is $\xi$-admissible. If  $\xi(j_2-1)=\xi(j_2+1)$ then an identical argument shows that 
  $(\bomega_{j_1,a_1},\bomega_{j_2,a_2}, \bomega(j_3,j_4))$ is $\xi$-admissible.
  Since the analogous equalities hold  for $\xi^*$   Lemma \ref{admis} now proves the irreduciblity of $[\bomega(j_1,j_2)]\otimes [\bomega(j_3,j_4)] $.

\end{pf}

\subsection{}\label{irred} In the rest of the section we shall prove the following theorem.
\begin{thm}\label{irredcrit} \hfill

\begin{enumerit}
\item[(a)] For all $\bpi\in\bp\bor_\xi$ and $k\in[1,n]$ we have $[\bpi][\bof_k]=[\bpi\bof_k]$.
\item[(b)]
Let   $\bpi_1, \bpi_2\in\bp\bor_\xi$.  The equality   $[\bpi_1\bpi_2]=[\bpi_1][\bpi_2]$ holds in $\cal K_0(\cal F_\xi)$ if
 one of the following conditions is  satisfied: 
 \begin{enumerit} 
 \item[(i)] $\bpi_1\bpi_2\notin\bp\bor_\xi$ and  $\max\bpi_s< \min\bpi_m$ for some  $s,m\in [1,2]$.
\item [(ii)] there exists $1\le i\le n$ and $a\in\{\xi(i)+1,\xi(i)-1\}$ such that 
$\bomega_{i,a}^{-1}\bpi_s\in\bp\bor_\xi$ for $s=1,2$,
\item[(iii)] there exists $s,m\in\{1,2\}$ such that either\\ (a)
  $\min\bpi_s<i=\min\bpi_m<j=\max\bpi_s<\max\bpi_m$ and $\Ht\bomega(i,j)$ is odd, or\\
 (b) $\min\bpi_s<i=\min\bpi_m<j=\max\bpi_m<\max\bpi_s$ and $\Ht\bomega(i,j)$ is even.
 \end{enumerit}
 \end{enumerit}
 \end{thm}

\subsection{Proof of Theorem \ref{irredcrit}}

 Notice that the hypothesis   of the theorem hold for the pair $(\bpi_1,\bpi_2)$  if and only if they hold for the pair $(\bpi_1^*,\bpi_2^*)$ of elements in $\bp\bor_{\xi^*}$. In particular if we show that the conditions imply that we can write $\bpi_1=\bomega_1\bomega_2$ so that $(\bomega_1,\bomega_2, \bpi_2)$ is $\xi$--admissible then the triple $(\bomega_1^*,\bomega_2^*,\bpi_2^*)$ is $\xi^*$--admissible. Lemma \ref{admis} then proves that $[\bpi_1]\otimes[\bpi_2]$ is irreducible. A similar comment applies to the pair $(\bof_p,\bpi)$. This observation  will be frequently used without further mention in the proof of the theorem.
 
 We proceed by induction on $\Ht\bpi$. If $\bpi=\bomega_{i,a}$ and $|\xi(i)-\xi(k)|= |k-i|$ the result was proved in \cite{HL2}. If $|\xi(i)-\xi(k)|\ne |k-i|$ then Proposition \ref{split} shows that the triple $(\bomega_{k,\xi(k)+1},\bomega_{k,\xi(k)-1},\bomega_{i,a})$ is $\xi$--admissible proving that  $[\bomega_{i,a}]\otimes[\bof_k]$ is irreducible.
If $\Ht\bpi>1$ write $\bpi=\bomega\bomega_{i,a}$ with $i=\max\bpi$ and $\bomega\in\bp\bor_\xi$. The inductive hypothesis and Proposition \ref{split} show that the triples $(\bomega, \bomega_{i,a},\bof_k)$ is $\xi$-admissible and the inductive step is proved.
\medskip

 \noindent All three assertions in part (b)  are proved by an  induction on $p=\Ht\bpi_1+\Ht\bpi_2$.  Proposition \ref{hl2result} shows that induction begins when $p\le 3$. It also shows that the result hold when $\Ht\bpi_1=\Ht\bpi_2=2$. Hence for the inductive step we  assume that the results hold for all $3\le p'<p$   and that either $\Ht\bpi_1>2$ or $\Ht\bpi_2>2$.
 \medskip
 
\noindent  To prove the inductive step for (i)  assume without loss of generality that $\Ht\bpi_1>2$ and  write  $\bpi_1=\bomega_1\bomega_{j_1,a_1}\bomega_{j_2,a_2}$ with $\bomega_1\in\bp\bor_\xi$  such that one of the following holds:
\begin{gather*}  \max\bpi_1<\min\bpi_2,\ \ \max\bomega_1<j_1<j_2=\max\bpi_1 \ \ {\rm{and}}\ \ \xi(j_1-1)=\xi(j_1+1),\ \\ \max\bpi_2<\min\bpi_1,\ \  \min\bpi_1=j_1<j_2<\min\bomega_1 \ \  {\rm{and}}\ \ \xi(j_2-1)=\xi(j_2+1).\end{gather*}
It follows that  $\bomega_1\bpi_2\notin\bp\bor_\xi$ and since $\bpi_1\bpi_2\notin\bp\bor_\xi$  we also have $\bomega_{j_1,a_1}\bomega_{j_2,a_2}\bpi_2\notin\bp\bor_\xi$. Hence  $[\bomega_{j_1,a_1}\bomega_{j_2,a_2}]\otimes[\bpi_2]$ and  $[\bomega_1]\otimes [\bpi_2]$  are irreducible by the inductive hypothesis.  Proposition \ref{split} shows that either $[\bomega_1]\otimes [\bomega_{j_1,a_1}\bomega_{j_2,a_2}]$ or $[\bomega_{j_1,a_1}\bomega_{j_2,a_2}]\otimes [\bomega_1]$ is $\ell$--highest weight proving  that the triple $(\bomega_1,\bomega_{j_1,a_1}\bomega_{j_2,a_2},\bpi_2)$ is $\xi$--admissible. The proof of the inductive step for (i) is complete. 

\medskip

\noindent To prove the inductive step for (ii) notice that
  the conditions on $\bpi_1$ and $\bpi_2$  imply that one of the following hold:   $\max\bpi_1=\min\bpi_2=i$ and $\xi(i-1)=\xi(i+1)$ or $\min\bpi_1=\min\bpi_2=i$ or $\max\bpi_1=\max\bpi_2=i$.  
 Assume  first that  $\max\bpi_1=\min\bpi_2=i$. If $\Ht\bpi_1\ge 3$ write $\bpi_1=\bomega_{k,c}\bomega_1$ with $\max\bomega_1=i$; otherwise $\Ht\bpi_2\ge 3$ write $\bpi_2=\bomega_{k,c}\bomega_2$ and $\min\bomega_2=i$. In the first case, Proposition \ref{split} and the inductive hypothesis show that the triple $(\bomega_{k,c},\bomega_1,\bpi_2)$ is $\xi$--admissible while in the second case we get that $(\bomega_{k,c},\bomega_2,\bpi_1)$ is  $\xi$-admissible. In either case the irreducibility of $[\bpi_1]\otimes[\bpi_2]$ follows from Lemma \ref{admis}. If $\min\bpi_1=\min\bpi_2$ assume without loss of generality that $\Ht\bpi_1\le\Ht\bpi_2$ and  let $k=\max\bpi_1$. Write $\bpi_2=\bomega\bomega'$ with $\bomega,\bomega'\in\bp\bor_\xi$ satisfying: $\min\bomega=\min\bpi_1$, $\max\bomega<k,$  $\min\bomega'\ge k$ and  $\min\bomega'=k$ if $\xi(k-1)=\xi(k+1)$.
 The inductive hypothesis applies to $(\bpi_1,\bomega)$, it also applies to  $(\bpi_1,\bomega')$ if $\xi(k-1)=\xi(k+1)$ and otherwise $\bpi_1\bomega'\notin\bp\bor_\xi$ and part (b)(i) applies and shows that the corresponding tensor products are irreducible.
  Since  Proposition \ref{split} applies to $(\bomega,\bomega')$ we have now shown that $(\bomega,\bomega',\bpi_1)$ is $\xi$--admissible and the inductive step is proved in this case

 \medskip

\noindent  Finally, we prove the inductive step for (iii). This amounts to proving the following: if 
$1\le i_1<i_2<i_3<i_4\le n$ then the tensor product $[\bomega(i_1,i_3)]\otimes [\bomega(i_2,i_4)]$ is irreducible if $\Ht\bomega(i_2,i_3)$ is odd and $[\bomega(i_1,i_4)]\otimes [\bomega(i_2,i_3)]$ is irreducible if $\Ht\bomega(i_2,i_3)$ is even. It is simple to see that $p=\Ht\bomega(i_1,i_3)+\Ht\bomega(i_2,i_4)=\Ht\bomega(i_1,i_4)+\Ht\bomega(i_2,i_3)$.
Since  Proposition \ref{hl2result} shows that the result holds when $p=4$ it  means that it holds whwn $\Ht\bomega(i_1,i_4)=2$. Hence for the inductive step we  may assume  $\Ht\bomega(i_1,i_4)\ge 3$.
\medskip

\noindent Suppose that $\Ht\bomega(i_2,i_3)=2$. Using Proposition \ref{split}, the inductive hypothesis and   parts (b)(i),(ii) of this theorem we see that one of the following holds:\\ 
$\bullet$ there exists $i_1<m\le i_2$ and $i_4>p\ge i_3$ such that $\bomega(i_1,i_4)=\bomega(i_1,m)\bomega(p,i_4)$  and  $(\bomega(i_1,m),\bomega(p,i_4), \bomega(i_2,i_3))$ is $\xi$--admissible,\\
$\bullet$ there exists $b\in\{\xi(i_4)+1,\xi(i_4)-1\}$ with  $\bomega(i_1,i_4)=\bomega(i_1,i_2)\bomega_{i_4,b}$ and $(\bomega(i_1,i_2),\bomega_{i_4,b}, \bomega(i_2,i_3))$ is $\xi$--admissible,\\
$\bullet$ there exists $a\in\{\xi(i_1)+1,\xi(i_1)-1\}$ with
 $\bomega(i_1,i_4)=\bomega_{i_1,a}\bomega(i_3,i_4)$
 and $(\bomega(i_3,i_4),\bomega_{i_1,a}, \bomega(i_2,i_3))$ is $\xi$--admissible.\\
 In all cases the irreducibility of $[\bomega(i_1,i_4)]\otimes[\bomega(i_2,i_3)]$ is proved. 
 \medskip

\noindent  Suppose that  $\Ht\bomega(i_2,i_3)\geq 3$ and  let $i_2< p< i_3$ be minimal such that $|\xi(p)-\xi(i_2)|=p-i_2$ with  $\xi(p-1)=\xi(p+1)$. Similarly, let $i_2< m< i_3$ be maximal so that $|\xi(i_3)-\xi(m)|= i_3-m$ and $\xi(m-1)=\xi(m+1)$. Then Proposition \ref{split}, parts (b)(i) and (b)(ii) and the inductive hypothesis show that one of the following hold:\\
$\bullet$
if  $\xi(i_2-1)=\xi(i_2+1)$, then
\begin{gather*}
    \Ht\bomega(i_2,i_3)\ {\rm{odd}}\ \implies  (\bomega(i_1,i_2),\bomega(p,i_3),\bomega(i_2,i_4))\ \ {\rm{is}}\  \xi-{\rm{admissible}},\\
     \Ht\bomega(i_2,i_3)\ {\rm{even}}\ \implies  (\bomega(i_1,i_2),\bomega(p,i_4),\bomega(i_2,i_3))\ \ {\rm{is}}\ \  \xi-{\rm{admissible}},\end{gather*} 
     $\bullet$
if $\xi(i_3-1)=\xi(i_3+1)$, then  
\begin{gather*}
\Ht\bomega(i_2,i_3)\ {\rm{odd}}\ \implies  (\bomega(i_2,m),\bomega(i_3,i_4),\bomega(i_1,i_3))\ \ {\rm{is}}\ \  \xi-{\rm{admissible}}\\
    \Ht\bomega(i_2,i_3)\ {\rm{even}}\ \implies \bomega(i_1,m),\bomega(i_3,i_4),\bomega(i_2,i_3))\ \  {\rm{is}}\  \xi-{\rm{admissible}},\\ \end{gather*}  
$\bullet$ 
 $\xi(i_j+1)\ne\xi(i_j-1)$ for $j=2,3$. 
 \medskip
 
\noindent  If $p=m$ there exists $b\in\bc(q)^\times$ such that $\bomega_{i_4,b}^{-1}\bomega(i_2,i_4)\in\bp\bor_\xi$ and the triple 
$(\bomega(i_2,p), \bomega_{i_4,b},\bomega(i_1,i_3))$ is $\xi$-admissible.

 \noindent If $p\ne m$ let $p<p'\leq m$ be minimum such that $\xi(p'-1)=\xi(p'+1)$.  Then
$(\bomega(i_2,p), \bomega(p',i_4),\bomega(i_1,i_3))$ is $\xi$-admissible if $\Ht\bomega(i_2,i_3)$ is odd and  otherwise $(\bomega(i_2,p),\bomega(p',i_3), \bomega(i_1,i_4))$ is $\xi$--admissible.
 \\ \\
In all cases the inductive step follows and the proof of the theorem is complete.

\section{Identities in $\cal K_0(\cal F_\xi)$}\label{reducible}

In this section we establish Proposition \ref{clusterrep} and Proposition \ref{clustermon}.

\subsection{} We will need  the converse of Theorem \ref{irredcrit}(b). The most elementary case is the following  well--known. Namely, let $i\le j$ satisfy $\xi(i)-\xi(j)=\pm(j-i)$; then  the following equality holds in $\cal K_0(\cal F_\xi)$:
\begin{equation}\label{basecase}[\bomega_{i,\xi(i)\pm 1}][\bomega_{j,\xi(j)\mp 1}]=[\bomega_{i,\xi(i)\pm 1}\bomega_{j,\xi(j)\mp 1}]+[\bomega_{i-1,\xi(i)}][\bomega_{j+1,\xi(j)}].\end{equation}
Given $\bpi=\bomega\bomega_{i,a}\in\bp\bor_\xi$ with $\bomega\in\bp\bor_\xi$, set $$\bpi'=\bomega\bomega_{i-1,\xi(i)},\ \ i=\max\bpi,\ \ '\bpi=\bomega_{i+1,\xi(i)}\bomega,\ \ i=\min\bpi.$$ 
In the remaining cases the converse is
 most conveniently stated as follows. 
\begin{thm}\label{t:reduc}
\hfill
\begin{enumerit}
\item[(i)]  Suppose that  $\bpi_1\bpi_2\in\bp\bor_\xi$  and $\max\bpi_1<\min\bpi_2$. Then \begin{equation}\label{disja}[\bpi_1][\bpi_2]=[\bpi_1\bpi_2] +[\bpi_1'][\, '\bpi_2].\end{equation}
\item[(ii)] Suppose that $ \bomega(m,p)\in\bp\bor_\xi$ 
and for $m<i<p$,  write   $$\bomega(m,i)=\bomega_1\bomega_{i,a},  \ \ \bomega(i,p)=\bomega_{i,b}\bomega_2.$$
If $a\ne b$ then \begin{gather*} (*)\ \ [\bomega(m,p)][\bomega_{i,a}]= [\bomega(m,i)][\bomega_2]+[\bomega_1']['\bomega(i,p)],\\ 
(**)\ \ [\bomega(m,p)][\bomega_{i,b}]= [\bomega_1][\bomega(i,p)]+[\bomega(m,i)']['\bomega_2].\end{gather*}
If $a=b$ then \begin{gather*}
(\dagger)\ \ [\bomega(m,p)][\bomega_{i,a'}] = [\bomega_1][\bof_i][\bomega_2] + [\bomega(m,i)']['\bomega(i,p)], \quad \{a,a'\} = \{\xi(i)+1, \xi(i)-1\}, \\
(\dagger\dagger)\ \  [\bomega(m,p)][\bomega_{i,a}]= [\bomega(m,i)][\bomega(i,p)]+[\bomega_1'][\bof_i]['\bomega_2].
\end{gather*}
Finally if $\bpi_1=\bomega_1\bomega_{i,a}$ and  $\bpi_2=\bomega_{i,b}\bomega_2$ are in $\bp\bor_\xi$ with  $\max\bpi_1=\min\bpi_2$ and $a\ne b$ then 
\begin{equation}\label{minmaxkr} 
[\bpi_1][\bpi_2]=[\bomega_1][\bof_i][\bomega_2]+[\bpi_1']['\bpi_2].\end{equation}
\item[(iii)] Assume that $i_1<i_2<i_3<i_4$  and write  $$\bomega(i_1, i_2)=\bomega_1\bomega_{i_2,a},\ \  \bomega(i_2,i_3)=\bomega_{i_2,b}\bomega\bomega_{i_3,c},\ \  \bomega(i_3,i_4)=\bomega_{i_3,d}\bomega_2.$$
Then$ (-1)^{\Ht(\bomegas(i_2,i_3))}\left([\bomega(i_1,i_3)][\bomega(i_2,i_4)]-[\bomega(i_1,i_4)][\bomega(i_2,i_3)]\right)$ is equal to $$[\bomega_1']^{\delta_{a,b}}[\bomega(i_1,i_2)')]^{1-\delta_{a,b}}\left(\prod_{s=i_2}^{i_3}[\bof_s]^{\delta_{\xi(s-1),\xi(s+1)}}\right)['\bomega_2]^{\delta_{c,d}}['\bomega(i_3,i_4))]^{1-\delta_{c,d}}.$$ 
\end{enumerit}
\end{thm}

From now on  we freely use (often without mention) the results of Theorem \ref{irredcrit}. We deduce Proposition \ref{clusterrep} and Proposition \ref{clustermon} before proving Theorem \ref{t:reduc}.

\subsection{Proof of Proposition \ref{clusterrep} } The proposition is obviously a special case of equation \eqref{basecase} and Theorem \ref{t:reduc}(i),(ii). However the translation from the formulation in this section to the one in Section 1 which is adapted to  cluster algebras needs some clarification which we provide for the readers convenience. We recall that $d_j=\delta_{\xi(j),\xi(j+2)}=\delta_{j,j_\diamond}$.

For part (i) of Proposition \ref{clusterrep} we take   $$\bpi_1 = \bomega_{i,\xi(i+1)},\qquad \bpi_2 = \bomega(i,i+1)^{1-d_i}\bomega_{i,\xi(i+1)\pm 2}^{d_i}=\bomega_{i, \xi(i+1)\pm 2}\bomega_{i+1,\xi(i+2)}^{1-d_i},
$$  where the second formula for $\bpi_2$ uses  the fact that  $\xi(i+2)=\xi(i+1)\mp 1=\xi(i)\mp 2$ if $d_i=0$.  Theorem \ref{irredcrit} gives $$[\bpi_1\bpi_2]=[\bof_i][\bomega_{i+1,\xi(i+2)}]^{1-d_i},\ \ \ \ [\bomega_{i+1,\xi(i)}\bomega_{i+1,\xi(i+2)}^{1-d_i}]= [\bof_{i+1}]^{1-d_i}[\bomega_{i+1,\xi(i)}]^{d_i},$$ 
Using  either \eqref{basecase} or \eqref{minmaxkr} we get $$[\bpi_1][\bpi_2]=[\bof_i][\bomega_{i+1,\xi(i+2)}]^{1-d_i}+[\bomega_{i-1,\xi(i)}][\bof_{i+1}]^{1-d_i}[\bomega_{i+1,\xi(i)}]^{d_i}$$ as needed.

 For   Proposition \ref{clusterrep}(ii)  using the definition of $\bar j$ we can rewrite its left hand side as
 $$[\bomega(i,\bar j)]^{1-\delta_{j,i_\diamond}}[\bomega_{i,\xi(i+1)\pm 2}]^{\delta_{j,i_\diamond}} = [\bomega(i,j+1)]^{1-d_j}\left([\bomega_{i,\xi(i+1)\pm 2}]^{1-\delta_{i,j_\bullet}}[\bomega(i, j_\bullet+1)]^{\delta_{i,j_\bullet}}\right)^{d_j}.$$ It is easiest to verify the four cases given by  $d_j\in\{0,1\}$ and $\delta_{j_\bullet,i}\in\{0,1\}$ separately. 
  If  $d_j=1$ the left hand side of Proposition \ref{clusterrep}(ii) is $[\bpi_1][\bpi_2]$ where  $$\bpi_1 = \bomega_{j,\xi(j+1)},\ \ \bpi_2= \bomega_{i,\xi(i+1)\pm 2}^{1-\delta_{i,j_\bullet}}\bomega(i,i+1)^{\delta_{i,j_\bullet}}  = \bomega_{i,\xi(i+1)\pm 2}\bomega_{i+1,\xi(i+2)\mp 2}^{\delta_{i,j_\bullet}}.
$$ 
and the right hand side is
$$ [\bof_j]^{d_{j-1}}\ 
[\bomega(i,j)]^ {1-d_{j-1}}  [\bomega_{i, \xi(i+1)\pm 2}]^{d_{j-1}} +    [\bof_i]^{\delta_{i,j_\bullet}} [\bomega_{j+1, \xi(j+2)}] [\bomega_{i-1, \xi(i)}]^{1-\delta_{i,j_\bullet}}.$$ 
Since $\delta_{j_\bullet,i}=0\implies j_\bullet<i\le j-1\implies d_{j-1}=0,$ and $\bomega(i,j)=\bomega_{i,\xi(i+1)\pm 2}\bomega_{j,\xi(j+1)}$ we see that the right hand side of Proposition \ref{clusterrep}(ii) is precisely the right  hand side of \eqref{basecase} and we are done. Otherwise $$\delta_{j_\bullet,i}=1\ \ {\rm and \ either}\ \ j_\bullet=j-1\ \ {\rm{or}} \ \ j_\bullet<j-1.$$
In the first case  $d_{j-1}=1$ and $i+1=j$ and so  the result follows from  \eqref{minmaxkr}; in the second case we have   $i+1<j$ and  $d_{j-1}=0$. Since $i=j_\bullet=i_\diamond$ we also  have $\xi(i)=\xi(i+2)$ which implies that $\bomega(i,i+1)\bomega_{j,\xi(j+1)}=\bomega(i,j)$, The result follows from Theorem \ref{t:reduc}(i).
\medskip

If $d_j=0$ then the left hand side of Proposition \ref{clusterrep} is $[\bpi_1][\bpi_2]$ where  $$\bpi_1=\bomega_{j,\xi(j)+1},\ \ \ \bpi_2= \bomega(i,j+1) = \bomega_{i,\xi(i+1)\pm 2}\bomega_{j,\xi(j+1)\mp 2}^{d_{j-1}}\bomega_{j+1,\xi(j+2)}$$  and  the right hand side of Proposition \ref{clusterrep} is $$ [\bof_j]^{d_{j-1}}\ 
[\bomega(i,j)]^ {1-d_{j-1}}  [\bomega_{i, \xi(i+1)\pm 2}]^{d_{j-1}} [\bomega_{j+1,\xi(j+2)}]+  [\bof_{j+1}]  [\bof_i]^{\delta_{i,j_\bullet}}  [\bomega_{i-1, \xi(i)}]^{1-\delta_{i,j_\bullet}}.$$
If $d_{j-1}=1$ then $i=j-1 = j_\bullet$ 
and the result follows from ($\dagger$) in Theorem \ref{t:reduc}(ii).
If  $d_{j-1}=0$, then the result follows from equation ($*$) in Theorem \ref{t:reduc}(ii).

The proof of part (iii) is a similar detailed analyses. Note that $[\bomega(i,\bar j)] = [\bomega(i, j_\bullet +1)]^{d_j}[\bomega(i,j+1)]^{1-d_j}$.
If $d_j=1$, we take $$\bpi_1 =\bomega_{j,\xi(j+1)},\ \ \bpi_2 =\bomega(i,j_\bullet+1)$$ and use  Theorem \ref{t:reduc}(i) if $j_\bullet+1 < j$ and \eqref{minmaxkr} if $j_\bullet+1 = j$. If $d_j = 0$ we take $$\bpi_1=\bomega_{j,\xi(j+1)},\ \ \ \bpi = \bomega(i,j+1).$$ Note that  $\bomega(i,j_\bullet+1)\bomega_{j,\xi(j+1)} = \bomega(i,j)\in \bp\bor$ if $j_\bullet+1<j$  and that if we write $k=(j_\bullet)_\bullet$ then  $$\bpi_2\bomega_{j,\xi(j+1)} = \begin{cases} (\bomega_{i,\xi(i+1)\pm 2}^{\delta_{k,i_\bullet}}\bomega(i,k+1)^{1-\delta_{k,i_\bullet}})\bof_j\bomega_{j+1,\xi(j+2)},\ \ d_{j-1}=1,\\ \\ \bomega(i,j_\bullet+1)\bomega_{j+1,\xi(j+2)},\ \ d_{j-1}=0.\end{cases}$$ An application of Theorem \ref{t:reduc} as in the other cases completes the proof.

\subsection{Proof of  Proposition \ref{clustermon}}. Let $\bomega,\bomega'\in \bp\bor$ be  such that $[\bomega\bomega']= [\bomega][\bomega']$. By Section \ref{iotaroot} we can choose $\alpha,\beta\in\Phi_{\ge -1}$  such that  $$[\bomega] = \iota(x[\alpha]),\ \ [\bomega']=\iota(x[\beta]).$$
We claim that $x[\alpha]x[\beta]$ is a cluster monomial. If not,  we can write  $$x[\alpha]x[\beta] = m_1x[\gamma]x[\eta] + m_2x[\gamma']x[\eta'],$$ where $m_1,m_2\in \bz_{\geq 0}[f_i: i\in I]$ where $\gamma,\gamma',\eta),\eta')$ are in $\Phi_{\ge -1}$.
Applying $\iota$ to both sides of the equation we get $$[\bomega\bomega']=[\bomega][\bomega']=\iota(x[\alpha]x[\beta])= \iota(m_1)\iota(x[\gamma])\iota(x[\eta]) + \iota(m_2)\iota(x[\gamma'])\iota(x[\eta']).$$ Since $\iota(m_1),\iota(m_2)\in \bz_{\geq 0}[\bof_i: i\in I]$ this means that we can write $[\bomega\bomega']$ as a nontrivial linear combination of elements $\{[\bpi]:\bpi\in\cal P^+_\xi\}$ which is absurd. 

Suppose now that $\bpi,\bpi'\in\bp\bor_\xi$ are  such that $[\bpi][\bpi']\ne[\bpi\bpi']$.  Theorem \ref{irredcrit} amd Theorem \ref{t:reduc}  imply  that $[\bpi\bpi']$ and  $[\bpi][\bpi']-[\bpi\bpi']$ are  in the $\bz_{\geq 0}[\bof_i: i\in I]$-  span of elements of the form  $[\bomega]=[\bomega_1]\cdots[\bomega_p]$ with $\bomega_s\in\bp\bor_\xi$, $1\le s\le p$. By the previous  discussion it follows that such products are the images of cluster monomials. Hence the inverse image under $\iota$ of  $[\bpi][\bpi']$ is a positive  linear combination of certain cluster monomials; in particular the inverse image is not a cluster monomial and the proof is complete.

\subsection{} In the rest of the paper we prove Theorem \ref{t:reduc}.  The crucial step is the following proposition whose proof we postpone to the next section.

\begin{prop}\label{csesf} 
 Let $\bomega_{i,a},\bpi$ be elements of $\bp\bor_\xi$ with $i< \min\bpi=j$ or $i> \max\bpi=k$ and $\bomega_{i,a}\bpi\in\bp\bor_\xi$. Let $b,c\in\bc(q)$ be such that $\bomega_{j,b}^{-1}\bpi$ and $\bomega_{k,c}^{-1}\bpi$ are elements of $\bp\bor_\xi$. We have $$[\bomega_{i,a}][\bpi]-[\bomega_{i,a}\bpi]=[\bomega_{i-1,\xi(i)}] ['\bpi],\ \ i<j,$$ $$
 [\bomega_{i,a}][\bpi]-[\bomega_{i,a}\bpi]=[\bomega_{i+1,\xi(i)}] [\bpi'],\ \ i>k.$$
 \end{prop}

\subsection{Proof of Theorem \ref{t:reduc}(i)} We need the following consequence of Proposition \ref{csesf}.

 \begin{lem}\label{krout} Let $\bomega_{i,a}, \bomega_{i,b}\bpi\in\bp\bor_\xi$ and assume that $a\ne b$ and $\min\bpi>i$  (resp. $\max\bpi<i$). Then $$[\bomega_{i,a}][\bomega_{i,b}\bpi]=[\bof_i][\bpi] + [\bomega_{i-1,\xi(i)}][\bomega_{i+1,\xi(i)}\bpi],$$ $$({\rm{resp.}}\ \ [\bomega_{i,a}][\bomega_{i,b}\bpi]=[\bof_i][\bpi] + [\bomega_{i+1,\xi(i)}][\bomega_{i-1,\xi(i)}\bpi].)$$ 
\end{lem}
\begin{pf}
 Proceed by induction on $\Ht\bomega_{i,b}\bpi$.  If $\Ht\bomega_{i,b}\bpi=1$ then the  result is well--known (see for instance \cite{H1}). Assume that we have proved the result if $\Ht\bomega_{i,b}\bpi<r$. Write $\bpi=\bomega_{m,c}\bomega$ with $m=\min\bpi$
 and note that $$a=\xi(i)\mp1\iff b=\xi(i)\pm 1\iff \xi(i+1)\pm 1 =\xi(i)\ {\rm{and}}\ \  c=\xi(m)\mp 1.$$ It follows that the pair $(\bomega_{i+1,\xi(i)},\bomega_{m,c}\bomega)$ satisfies  the conditions of Proposition \ref{csesf} if $i+1\ne m$ and the    inductive hypothesis of this Lemma if $i+1=m$ and so we have,  
$$(*)\ \ [\bomega_{i+1,\xi(i)}][\bomega_{m,c}\bomega] = [\bomega_{i+1, \xi(i)}\bomega_{m,c}\bomega] + [\bomega_{i,a}][\bomega_{m+1,\xi(m)}\bomega].$$ 
 The inductive hypothesis  and Proposition \ref{csesf} also give
 \begin{eqnarray*}&[\bomega_{i,a}][\bomega_{i,b}][\bpi]&=  [\bomega_{i,a}]\left([\bomega_{i,b}\bpi]   +[\bomega_{i-1,\xi(i)}][\bomega\bomega_{m+1,\xi(m)}]\right)\\&&=\left([\bof_i]+[\bomega_{i-1,\xi(i)}][\bomega_{i+1,\xi(i)}]\right)[\bpi]\\&&=[\bof_i][\bpi]+[\bomega_{i-1,\xi(i)}]\left([\bomega_{i+1, \xi(i)}\bomega_{m,c}\bomega] + [\bomega_{i,a}][\bomega_{m+1,\xi(m)}\bomega]\right).
 \end{eqnarray*}
 Equating the first and third terms on the right hand side and using (*) gives $$[\bof_i][\bpi] + [\bomega_{i- 1,\xi(i)}][\bomega_{i+1 ,\xi(i)}\bomega_{m,c}\bomega]=[\bomega_{i,a}][\bomega_{i,b}\bpi],$$ which establishes the inductive step. 
\end{pf}
The proof of Theorem \ref{t:reduc}(i) proceeds by an induction on $\Ht\bpi_1$ with Proposition \ref{csesf} showing that induction begins when $\Ht\bpi_1=1$. 
For the inductive step, recall that $\bpi_1=\bomega_{i,a}\bomega_1$ and $\bpi_2=\bomega_{j,b}\bomega_2$ with $\max\bpi_1=i< j=\min\bpi_2$. 
Since  $\bpi_1\bpi_2\in\bp\bor_\xi$ we see  that Proposition \ref{csesf} applies to the pairs $(\bomega_1,\bomega_{i,a})$,  $(\bomega_{i,a},\bpi_2)$ and also to the pairs $(\bomega_{i+1,\xi(i)}, \bpi_2)$ and $(\bomega_1, \bomega_{i-1,\xi(i)})$ if $i+1\ne j$ and $i-1\ne \min\bomega_1$. If $i+1=j$ (resp. $i-1=\min\bomega_1$) then Lemma \ref{krout} applies to $(\bomega_{i+1,\xi(i)}, \bpi_2)$ (resp.$(\bomega_1, \bomega_{i-1,\xi(i)})$).  Together with the inductive hypothesis which applies to  $(\bomega_1,\bomega_{i,a}\bpi_2)$ we get the following series of equalities:

\begin{gather*}
[\bpi_1][\bpi_2]+[\bomega_1']\left([\bomega_{i+1,\xi(i)}\bpi_2]+[\bomega_{i,\xi(i+1)}]['\bpi_2]\right)\ \ \ 
\\=\left([\bpi_1+[\bomega_1'][\bomega_{i+1,\xi(i)}]\right)[\bpi_2]\\=[\bomega_1][\bomega_{i,a}][\bpi_2]\\=[\bomega_1]\left([\bomega_{i,a}\bpi_2]+[\bomega_{i-1,\xi(i)}]['\bpi_2]\right)\\ =[\bpi_1\bpi_2]+[\bomega_1'][\bomega_{i+1,\xi(i)}\bpi_2]+[\bomega_1\bomega_{i-1\xi(i)}]['\bpi_2]+ [\bomega_1'][\bomega_{i,\xi(i-1)}]['\bpi_2].
\end{gather*}
Equating the first and the fifth terms gives the inductive step since $\xi(i-1)= \xi(i+1)$ and part (i) is proved.
\subsection{Proof of Theorem \ref{t:reduc}
(ii)}

Suppose that  $a\neq b$ which means that  $\xi(i-1)\neq \xi(i+1)$ and hence $\bomega(m,p)=\bomega_1\bomega_2$.  We prove equation (**); the proof of (*) being an obvious modification.
Using   Theorem \ref{irredcrit}(b)(i)  gives that $[\bomega(m,p)\bomega_{i,b}]=[\bomega_1][\bomega(i,p)]$
and we have to prove that $$[\bomega(m,p)][\bomega_{i,b}]-[\bomega_1][\bomega(i,p)]=[\bomega(m,i)']['\bomega_2].$$ For this we calculate $[\bomega_1][\bomega_{i,b}][\bomega_2]$ in two ways by using Proposition \ref{csesf} on $(\bomega_{i,b},\bomega_2)$ and part (i) of the theorem on  $(\bomega_1,\bomega_2)$. This gives, \begin{eqnarray*}  &[\bomega_1][\bomega_{i,b}][\bomega_2]&= [\bomega_1][\bomega(i,p)]+[\bomega_1][\bomega_{i-1,\xi(i)}]['\bomega_2]\\&&  =[\bomega_{i,b}][\bomega(m,p)]+ [\bomega_{i,b}][\bomega_1']['\bomega_2].\end{eqnarray*}
Equating we see that we must prove that   \begin{equation}\label{used} [\bomega_1][\bomega_{i-1,\xi(i)}]-[\bomega_1'][\bomega_{i,b}]=[\bomega_1\bomega_{i-1,\xi(i)}]=[\bomega(m,i)'].\end{equation} This follows since  Proposition \ref{csesf} applies if $\min\bomega_1<i-1$ (and Lemma \ref{krout} if $\min\bomega_1=i-1$) to the pair $(\bomega_1,\bomega_{i-1,\xi(i)})$.

If $a=b$ then  $\xi(i-1)=\xi(i+1)$ and hence Theorem \ref{irredcrit} shows that $[\bomega(m,i)][\bomega(i,p)]=[\bomega(m,p)\bomega_{i,a}]$ and $[\bomega(m,p)\bomega_{i,a'}] = [\bomega_1][\bof_i][\bomega_2]$. To prove $(\dagger)$
we use  part (i) of the theorem on the pair $(\bomega_1,\bomega(i,p))$ and Lemma \ref{krout} on the pair $(\bomega_{i,a'}, \bomega(i,p))$ to get
\begin{eqnarray*}&[\bomega_1][\bomega_{i,a'}][\bomega(i,p)]&=[\bomega_{i,a'}][\bomega(m,p)]+[\bomega_{i,a}][\bomega_1']['\bomega(i,p)],\\&&=[\bomega_1][\bof_i][\bomega_2]+[\bomega_1][\bomega_{i-1,\xi(i)}]['\bomega(i,p)].
\end{eqnarray*}
 Equating the right hand sides and using \eqref{used} gives the result.
The proof of $(\dagger\dagger)$ is similar;
 we calculate  $[\bomega_{i,a}][\bomega_1][\bomega(i,p)]$ in two ways by using Proposition \ref{csesf} on $(\bomega_{i,a},\bomega_1)$ and part (i) of the theorem on $(\bomega_1,\bomega(i,p))$. This gives 
\begin{eqnarray*}&[\bomega_{i,a}][\bomega_1][\bomega(i,p)]&=[\bomega(m,i)][\bomega(i,p)] +[\bomega_1'][\bomega_{i+1,\xi(i)}] [\bomega(i,p)]\\&&=[\bomega_{i,a}][\bomega(m,p)] +[\bomega_{i,a}][\bomega_1']['\bomega(i,p)]
\end{eqnarray*}
We then observe that  $\bomega_{i+1,\xi(i)}\bomega_2\in\bp\bor_\xi$ if $\xi(i)\ne\xi(i+2)$ and 
$\bomega_{i+1,\xi(i)}\bomega_2 = \bof_{i+1}\bomega'$, with $\bomega'\in\bp\bor_\xi$, if $\xi(i+1)=\xi(i+3)$. Then we can apply the results proved above of part (ii) of this theorem to the
pair  $(\bomega_{i+1,\xi(i)}, \bomega(i,p))$, and hence either by $(**)$ or by $(\dagger)$ we get
 $$[\bomega_{i+1,\xi(i)}][\bomega(i,p)] = [\bomega_{i,a}][\bomega_{i+1,\xi(i)}\bomega_2] + ['\bomega_2][\bof_i]$$
Equation $(\dagger\dagger)$ now follows  by a substitution, recalling that $['\bomega(i,p)]=[\bomega_{i+1,\xi(i)}\bomega_2]$, by definition. 

Finally we prove \eqref{minmaxkr}.  If $\Ht\bpi_1=1$ or $\Ht\bpi_2=1$ this was proved in Lemma \ref{krout}. Hence we may assume that $\bpi_1=\bomega(m,i)$ and $\bpi_2=\bomega(i,p)$ for some $m<i<p$.
Since $a\ne b$ we use Theorem \ref{irredcrit} to see that  $[\bomega(m,i)\bomega(i,p)] = [\bof_i][\bomega(m,p)]$ and  we prove that 
\begin{equation}\label{ekrsimp} [\bomega(m,i)][\bomega(i,p)] - [\bof_i][\bomega(m,p)] = [\bomega(m,i)']['\bomega(i,p)].
\end{equation}
For this we note that $\bomega_1\bomega_2 = \bomega(m,p)$ and hence, using part (i) of the theorem to the pair $(\bomega_1,\bomega_2)$ and Proposition \ref{csesf} or Lemma \ref{krout} to the pairs $(\bomega_1,\bomega_{i-1,\xi(i)})$ and $(\bomega_{i,a},\bomega(i,p))$ we get
\begin{gather*}  
[\bof_i]\left([\bomega(m,p)] +[\bomega_1']['\bomega_2]\right)+ \left( [\bomega(m,i)'] +[\bomega_1'][\bomega_{i,b}]\right)['\bomega(i,p)] \\
=[\bomega_1][\bof_i][\bomega_2] + [\bomega_{1}][\bomega_{i-1,\xi(i)}]['\bomega(i,p)]  = [\bomega_1][\bomega_{i,a}][\bomega(i,p)]\\ =
[\bomega(m,i)][\bomega(i,p)] + [\bomega_1'] [\bomega_{i+1,\xi(i)}][\bomega(i,p)].
\end{gather*}
Equating the first and last terms we see that \eqref{ekrsimp} follows  if we prove that
$$[\bof_i][\bomega_2']+[\bomega_{i,b}]['\bomega(i,p)] =[\bomega_{i+1,\xi(i)}][\bomega(i,p)].$$ But this follows from the cases of part (ii) of this theorem proved above. This completes the proof of part (ii).

\subsection{Proof of Theorem \ref{t:reduc}(iii)}\label{reducib}  
 We proceed by induction on $N=\Ht\bomega(i_1,i_3)+\Ht\bomega(i_2,i_4)$ with \cite{HL2} showing that induction begins when $N=4$.   Recall that $$\bomega(i_1, i_2)=\bomega_1\bomega_{i_2,a},\ \  \bomega(i_2,i_3)=\bomega_{i_2,b}\bomega\bomega_{i_3,c},\ \  \bomega(i_3,i_4)=\bomega_{i_3,d}\bomega_2.$$
 Set $$[\bof_{i_2,i_3}]=\prod_{s=i_2}^{i_3}[\bof_s]^{\delta_{\xi(s-1),\xi(s+1)}},$$ and note that $\bomega(i_2,i_4)=\bomega_{i_2,b}\bomega\bomega_{i_3,c}^{\delta_{c,d}}\bomega_2$.
 
 {\bf Case 1: $a=b$ or $c=d$} Suppose that $a=b$; the proof is similar when $c=d$. Then $\bomega_1\bomega(i_2,i_s)\in\bp\bor_\xi$ for $s=3,4$. Hence 
 \eqref{disja} gives
 \begin{eqnarray*}
 &[\bomega_1][\bomega(i_2,i_4)][\bomega(i_2,i_3)]&=
\left([\bomega(i_1,i_4)] +[\bomega_1']['\bomega(i_2,i_4)]\right)[\bomega(i_2,i_3)]\\ &&=
\left([\bomega(i_1,i_3)] + [\bomega_1']['\bomega(i_2,i_3)])\right)[\bomega(i_2,i_4)].
 \end{eqnarray*} 
 The result  follows if we prove that
\begin{equation*}\label{1} ['\bomega(i_2,i_3)][\bomega(i_2,i_4)] - ['\bomega(i_2,i_4)][\bomega(i_2,i_3)]=(-1)^{\Ht\bomegas(i_2,i_3)}[\bof_{i_2,i_3}]['\bomega_2]^{\delta_{c,d}}['\bomega(i_3,i_4))]^{1-\delta_{c,d}}.\end{equation*}
Note that we have the following possibilities for the pair $('\bomega(i_2,i_3), '\bomega(i_2,i_4))$:
$$(\bomega(i_2+1,i_3),\bomega(i_2+1,i_4)),\ \ (\bof_{i_2+1}\bomega(m,i_3), \bof_{i_2+1}(\bomega(m,i_4)),$$
$$(\bof_{i_2+1}\bomega_{i_3,c},\  \bof_{i_2+1}\bomega_{i_3,c}^{\delta_{c,d}}\bomega_2),\ \ (\bof_{i_2+1}, (\bof_{i_2+1}\bomega_2)^{\delta_{c,d}}\bomega(i_2+1,i_4)^{1-\delta_{c,d}}).$$
In the first case,  $\Ht\bomega(i_2,i_3)=\Ht\bomega(i_2+1,i_3)$,  $\Ht \bomega(i_2, i_3)<\Ht\bomega(i_1,i_3)$ the inductive hypothesis applies to $i_2<i_2+1<i_3<i_4$ and gives the result.  In the second case the inductive hypothesis applies to $i_2<m<i_3<i_4$ and gives the result. In the third case we use equations $(*)$ and  $(\dagger\dagger)$ of Theorem \ref{t:reduc}(ii) to get the result. In the fourth case we use Theorem \ref{t:reduc}(i) if $c=d$ and \eqref{minmaxkr} if $c\ne d$.

{\bf Case 2}. Assume that $a\ne b$ and $c\ne d$. Since $N\ge 5$ we may assume without loss of generality that $\Ht\bomega(i_1,i_3)\ge 3$.
If $\Ht\bomega(i_1,i_2)\ge 3$  let $i_1<j<i_2$ be minimal with $\xi(j-1)=\xi(j+1)$. We choose $z\in\bc(q)^\times$ so that   $\bomega(i_1,i_2)\bomega_{i_1,z}^{-1}\in\bp\bor_\xi$ and calculate $[\bomega_{i_1,z}]\left([\bomega(j,i_4)][\bomega(i_2,i_3)]-[\bomega(j,i_3)][\bomega(i_2,i_4)]\right)$ in two ways to  get two expressions for it; the first one by using the inductive hypothesis which shows that it is equal to $$(-1)^{\Ht\bomegas(i_2,i_3)}[\bomega_{i_1,z}][\bomega(j,i_2)'][\bof_{i_2,i_3}]['\bomega(i_3,i_4)]$$ and the second by using Theorem \ref{t:reduc}(i) on the pairs $(\bomega_{i_1,z},\bomega(j,i_s))$, $s=3,4$  which gives that it is equal to
\begin{eqnarray*}[\bomega(i_1,i_4)][\bomega(i_2,i_3)]-[\bomega(i_1,i_3)][\bomega(i_2,i_4)]+[\bomega_{i_1-1,\xi(i_1)}](['\bomega(j,i_4)][\bomega(i_2,i_3)]-['\bomega(j,i_3)][\bomega(i_2,i_4)]
\end{eqnarray*}
Hence the inductive step follows if we prove that 
\begin{gather*} \left([\bomega_{i_1,z}][\bomega(j,i_2)']-[\bomega(i_1,i_2)']\right)[\bof_{i_2,i_3}]['\bomega(i_3,i_4)]=\\ (-1)^{\Ht\bomegas(i_2,i_3)}[\bomega_{i-1,\xi(i)}]\left(['\bomega(j,i_4)][\bomega(i_2,i_3)]-['\bomega(j,i_3)][\bomega(i_2,i_4)]\right).\end{gather*}
This is proved by noting that $$\bomega(j,i_2)'=\bomega(j,i_2-1)^{(1-\delta_{j,i_2-1})(1-\delta_{\xi(i_2),\xi(i_2)-2})}(\bomega_3^{(1-\delta_{j,i_2-1})}\bof_{i_2-1})^{\delta_{\xi(i_2),\xi(i_2)-2}},
$$
 where $\bomega_1 = \bomega_3\bomega_{i_2-1,\xi(i_2)\pm 2}$ if $\xi(i_2)=\xi(i_2-2)$ and considering the different cases.  In each case, Theorem \ref{t:reduc}(i) applies to the left hand side while the induction hypothesis or Theorem \ref{t:reduc}(ii) applies to the right hand side and gives the answer. As an example suppose that $j=i_2-1$ and $\xi(i_2-2)=\xi(i_2)$. Then $\bomega(j,i_2)'=\bof_{j}$ and the minimality of $j$ shows that $\bomega(i_1,i_2)'=\bomega_{i_1,z}\bof_j$ and hence the left hand side is zero.  On the right hand side  since $\xi(i_2-1)\ne\xi(i_2+1)$ by assumption we get $'\bomega(j,i_s)=\bomega(i_2,i_s)$ and so the right hand side is zero as well. We omit the details in other cases.

Finally suppose that $j>i_2$  and  let $b'\in\bc(q)$ be such that $\{b,b'\} = \{\xi(i_2)+ 1,\xi(i_2)-1\}$; we have the following series of equalities.
\begin{gather*}
\left([\bomega(i_1,i_4)]+ [\bomega_{i_1-1,\xi(i_1)}]['\bomega(j,i_4)]\right) [\bomega(i_2,i_3)] + [\bomega(i_1,i_2)']['\bomega(j,i_3)][\bomega(j,i_4)]\\
=\left([\bomega_{i_1,a}]
[\bomega(i_2,i_3)]+ [\bomega(i_1,i_2)']['\bomega(j,i_3)]\right)[\bomega(j,i_4)]\\= [\bomega_{i_2,b'}][\bomega(i_1,i_3)][\bomega(j,i_4)]\\
= \left([\bomega(i_2,i_4)] + [\bomega_{i_2-1,\xi(i_2)}]['\bomega(j,i_4)]\right)[\bomega(i_1,i_3)]\\
=[\bomega(i_2,i_4)][\bomega(i_1,i_3)] + ['\bomega(j,i_4)]\left([\bomega(i_1,i_2)'][\bomega(j,i_3)] + [\bomega_{i_1-1,\xi(i_1)}][\bomega(i_2,i_3)]\right).
\end{gather*}
 where the first and third equality follow from applying \eqref{disja} to the pairs $(\bomega_{i_1,a},\bomega(j,i_4))$ and $(\bomega_{i_2,b'},\bomega(j,i_4))$, respectively, and the second and fourth equality follow  busing $(*), (**) $ of Theorem \ref{t:reduc}(ii) to $ (\bomega_{i_2,b'},\bomega(i_1,i_3))$ and $(\bomega_{i_2-1,\xi(i_2)},\bomega(i_1,i_3))$.  The inductive step follows by establishing 
 \begin{equation}\label{case3}(-1)^{\Ht\bomega(i_2,i_3)}[\bof_{i_2,i_3}]['\bomega(i_3,i_4)] = [\bomega(j,i_3)]['\bomega(j,i_4)] - ['\bomega(j,i_3)][\bomega(j,i_4)].
 \end{equation}
 The calculations are similar to the ones done so far and we omit further details.

\section{Proof of Proposition \ref{csesf}}\label{starti}

In this section we prove Proposition \ref{csesf} when $i<j$; the proof in the case $i>k$ is identical. We recall the statement of the proposition for the readers convenience.

\begin{prop}\label{csesfa} 
 Suppose that $\bomega_{ia}\bomega_{j,b}\bomega\in\bp\bor_\xi$ with $i< j<\min\bomega$ and set $\bpi=\bomega_{j,b}\bomega$.  We have  $$[\bomega_{i,a}][\bpi]-[\bomega_{i,a}\bpi]=[\bomega_{i-1,\xi(i)}] [\bomega_{j+1,\xi(j)}\bomega].$$
 \end{prop}
 
 We make some preliminary remarks about the proof.
Recall from Lemma \ref{prime} that  for all $\bomega\in\bp\bor_\xi$ the module $[\bomega]$ is prime, i.e., that it cannot be written as a tensor product of non-trivial finite--dimensional representations of $\hat\bu_q$. It follows that the module $[\bomega_{i,a}\bpi]$ is a proper subquotient of $[\bomega_{i,a}]\otimes[\bpi]$.
 \medskip
 
 We claim that  $[\bomega_{i-1,\xi(i)}] \otimes [\bomega_{j+1,\xi(j)}\bomega] $ is irreducible. The condition  $\bomega_{i,a}\bomega_{j,b}\bomega\in\bp\bor_\xi$ forces   $\xi(j-1)=\xi(j+1)$ and hence   $$j+1<\min\bomega\implies \bomega_{i-1,\xi(i)}\bomega_{j+1,\xi(j)}\bomega\notin\bp\bor_\xi$$ and the claim follows from Theorem \ref{irredcrit}(b)(i). Otherwise, we have  $$j+1=\min\bomega\implies \bomega_{j+1,\xi(j)}\bomega=\bof_{j+1}\bomega',\ \ \bomega'\in\bp\bor_\xi\cup\{1\},$$   Theorem \ref{irredcrit} (a) gives  $$[\bomega_{j+1,\xi(j)}\bomega']=[\bomega']\otimes [\bof_{j+1}],\ \ [\bomega_{i-1,\xi(i)}]\otimes [\bof_{j+1}]=[\bomega_{i-1,\xi(i)}\bof_{j+1}],$$
 while  (b)(i), $$[\bomega_{i-1,\xi(i)}]\otimes[\bomega']=[\bomega_{i-1,\xi(i)}\bomega'].$$ An application of Theorem \ref{pairs} now proves the claim in this case.
 \medskip
 
 In the first part of this section we shall show that $[\bomega_{i-1,\xi(i)}] \otimes [\bomega_{j+1,\xi(j)}\bomega] $  is also a subquotient of $[\bomega_{i,a}]\otimes[\bomega_{j,b}\bomega]$; in particular  $$\dim[\bomega_{i,a}]\dim[\bomega_{j,b}\bomega]\ge \dim[\bomega_{i,a}\bomega_{j,b}\bomega]+\dim [\bomega_{i-1,\xi(i)}] \dim [\bomega_{j+1,\xi(j)}\bomega].$$
The proposition clearly follows if we prove the reverse inequality. This is done by  using  a presentation of the graded limit of the modules $[\bpi]$, $\bpi\in\bp\bor_\xi$  given 
 in \cite{BCM} along with  some additional results in the representation theory of current algebras.

 \subsection{}\label{lb} The proof of the next result is   an elementary application of $q$--character theory for quantum affine algebras. \begin{lem} The module $[\bomega_{i-1,\xi(i)}\bomega_{j+1,\xi(j)}\bomega]$ occurs in the Jordan--Holder series of $[\bomega_{i,a}]\otimes [\bomega_{j,b}\bomega].$\end{lem}
 \begin{pf} It suffices to show that there exists an $\ell$--highest weight vector with $\ell$--highest weight  $\bomega_{i-1,\xi(i)}\bomega_{j+1,\xi(j)}\bomega$ in $[\bomega_{i,a}]\otimes [\bpi]$ (resp. $[\bpi]\otimes [\bomega_{i,a}]$) if $a=\xi(i)+1$ (resp. $a=\xi(i)-1$). \vskip6pt
 
 \noindent But this is true by a routine argument using $q$--characters.  Namely one observes that the element $\bomega_{i-1,\xi(i)}\bomega_{j+1,\xi(j)}\bomega$ is an $\ell$-weight of $[\bomega_{i,a}]\otimes [\bpi]$ but not of $[\bomega_{i,a}\bpi]$.  It is then elementary \vskip 6pt
 
 \noindent to see  that the corresponding eigenvector is necessarily highest weight. We omit the details.
 \end{pf}

\subsection{} We need some standard notation from the theory of  simple Lie algebras. Thus, $\lie h$  denotes a Cartan subalgebra of $\lie{sl}_{n+1}$,   $\{\alpha_i: 1\le i\le n\}$  a set of simple roots for $(\lie{sl}_{n+1},\lie h)$ and  $R^+=\{\alpha_{i,j}:=\alpha_i+\cdots+\alpha_j: 1\le i\le j\le n\}$  the corresponding set of positive roots.  Fix a Chevalley basis $x_{i,j}^\pm$, $1\le i\le j\le n$ and $h_j$, $1\le j\le n$  for $\lie{sl}_{n+1}$. Set $x_{j,j}^\pm=x_j^\pm$ and $h_{i,j}=h_i+\cdots+h_j$ for all $1\le i\le j\le n$.
\medskip

As in the earlier sections $P^+$ will be the set of dominant integral weights corresponding to a set $\{\omega_i: 1\le i\le n\}$ of fundamental weights and we set $$P^+(1)=\{\lambda\in P^+:\lambda(h_i)\le 1, \ 1\le i\le n\}.$$
For $\lambda\in P^+$ let $V(\lambda)$ be an  irreducible  finite dimensional $\lie {sl}_{n+1}$ with highest weight $\lambda$.

Let $t$ be an indeterminate and   $\bc[t]$ the corresponding polynomial ring with complex coefficients. 
Denote by  $\lie{sl}_{n+1}[t]$ the Lie algebra  with underlying vector space $\lie{sl}_{n+1}\otimes \bc[t]$ and commutator given by $$[a\otimes f, b\otimes g]=[a,b]\otimes fg,\ \ a,b\in\lie{sl}_{n+1},\ \ f,g\in\bc[t].$$ Then $\lie{sl}_{n+1}[t]$ and its universal enveloping algebra admit a natural $\bz_+$-grading given by declaring a monomial $(a_1\otimes t^{r_1})\cdots (a_p\otimes t^{r_p})$ to have grade $r_1+\cdots +r_p$, where $a_s\in\lie{sl}_{n+1}$ and $r_s\in\bz_+$ for $1\le s\le p$.
 
 \subsection{} We shall be interested in the category of $\bz_+$--graded modules for $\lie{sl}_{n+1}[t]$. An object of this category is a  module $V$ for $\lie{sl}_{n+1}[t]$ which admits a compatible  $\bz$--grading, i.e., $$V=\bigoplus_{s\in\bz} V[s],\ \ (x\otimes t^r)V[s]\subset V[r+s],\ \ x\in\lie{sl}_{n+1},\ \ r\in\bz_+.$$ For any $p\in\bz$ we let $\tau_p^*V$ be the graded module which  given by shifting the grades  up by $p$ and leaving the action of $\lie{sl}_{n+1}[t]$ unchanged.
The morphisms between graded modules are $\lie{sl}_{n+1}[t]$- maps  of grade zero.  A $\lie{sl}_{n+1}$--module $M$  will  be regarded as an object  (denoted $\ev_0^*M$) of this category by placing $M$ in degree zero  and requiring that $$(a\otimes t^r) m=\delta_{r,0}am,\ \ a\in\lie{sl}_{n+1},\ \ m\in M\ \ r\in\bz_+.$$

\noindent  For $\lambda\in P^+$, the local Weyl module $W_{\loc}(\lambda)$ is the $\lie{sl}_{n+1}[t]$--module generated by an element $w_\lambda$ with graded  defining relations:  \begin{equation}\label{localweyld} (x_i^+\otimes 1)w_\lambda=0,\ \ (h\otimes t^r)w_\lambda=\delta_{r,0}\lambda(h)w_\lambda,\ \ \  (x_i^-\otimes 1)^{\lambda(h_i)+1}w_\lambda=0,\end{equation}  where  $1\le i\le n, $ and $r\in\bz_+$. Define a grading on $W_{\loc}(\lambda)$ by requiring $\gr w_\lambda=0$. 
 It is straightforward to see that $$W(\omega_i)\cong_{\lie{sl}_{n+1}} V(\omega_i),\ \ 1\le i\le n.$$
 In general $W_{\loc}(\lambda)$ has a unique graded irreducible quotient which is isomorphic to  $\ev_0^*V(\lambda)$. It is obtained by imposing the additional relation $(x^-_\alpha\otimes t)w_\lambda=0$ for all $\alpha\in R^+$.
 
\subsection{}   Given $\mu\in P^+(1)$ set  \begin{gather*}\min\mu=\min\{i:\mu(h_i)=1\},\\ 
 R^+(\mu)=\{\alpha_{i,j}\in R^+: 1\le i<j\le n,\ \  \mu(h_i)=1=\mu(h_j)\ \ {\rm{and}}\ \ \mu(h_{i,j})=2\}.\end{gather*} Given $\lambda=2\lambda_0+\lambda_1\in P^+$ with $\lambda_0\in P^+$ and $\lambda_1\in P^+(1)$ and $0\le i<\min\lambda_1$, define
  $M(\omega_i,\lambda)$ to be the graded $\lie{sl}_{n+1}[t]$-module  generated by an element $m_{i,\lambda}$ of grade zero   satisfying the graded relations in \eqref{localweyld} and  \begin{equation}\label{mil} (x^-_p\otimes t^{(\lambda_0+\lambda_1+\omega_i)(h_p)})m_{i,\lambda}=0=(x_\alpha^-\otimes  t^{\lambda_0(h_\alpha)+1})m_{i,\lambda},\ \ 1\le p\le n,\ \alpha\in R^+(\lambda_1).\end{equation}
  Clearly $M(\omega_i,\lambda)$ is a graded quotient of $W_{\loc}(\lambda)$ 
 and $$M(0,\omega_i)\cong_{\lie {sl}_{n+1}[t]} M(\omega_i,0)\cong_{\lie {sl}_{n+1}[t]} W_{\loc}(\omega_i)\cong_{\lie sl_{n+1}} V(\omega_i).$$ 
 If $\lambda_1\ne 0$ and $i_1=\min\lambda_1$, then $R^+(\lambda_1+\omega_i)= R^+(\lambda_1)\cup\{\alpha_{i,i_1}\}$ and it  is simple to check  that the assignment $m_{i,\lambda}\to m_{0,\lambda+\omega_{i}}$ gives rise to the following short exact sequence of $\lie {sl}_{n+1}[t]$--modules \begin{equation}\label{ses1}0\to \bu(\lie g[t])(x_{i, i_1}^-\otimes t^{\lambda_0(h_{i,i_1})+1})m_{i,\lambda}\to  M(\omega_i,\lambda)\to M(0,\lambda+\omega_{i})\to 0.\end{equation}
 The modules $M(0,\lambda)$, $\lambda\in P^+$  are examples of  level two Demazure modules; the latter  have been studied extensively and  are usually denoted as $D(2,\lambda)$ in the literature. We now state a  result which relates modules for the quantum affine algebra which are defined over $\bc(q)$ and modules for $\lie{sl}_{n+1}[t]$ which are defined over $\bc$. Denote by $\dim_{\bc(q)} V$ the dimension of a module $V$  for the quantum affine algebra and by $\dim M$ the dimension over $\bc$  of a module $M$ for $\lie{sl}_{n+1}[t]$.
 
 Part (i)  of the following result was proved in  \cite[Theorem 1]{CSVW} and parts (ii) and (iii) were proved in \cite[Theorem 1]{BCM}.  
\begin{thm}\label{dimprod}\label{bcm}
\begin{enumerit}
\item[(i)] Let $\mu\in P^+(1)$, $\nu_1, \nu\in  P^+$ with $\nu-\nu_1\in P^+$. Then \begin{eqnarray*}&\dim M(0, 2\nu)\dim M(0,\mu)&= \dim M(0, 2\nu+\mu)\\&&=\dim M(0, 2\nu_1)\dim M(0, 2(\nu-\nu_1)+\mu).\end{eqnarray*}
\item[(ii)]  Let $\xi: I\to\bz$ be an arbitrary height function and $\bpi\in\bp\bor_\xi$. We have $$\dim  M(0, \wt\bpi)=\dim_{\bc(q)} [\bpi].$$
\item[(iii)] For all $1\le p\le n$ we have $\dim M(0,2\omega_p)=\dim_{\bc(q)}[\bof_p]$.\end{enumerit} \hfill\qedsymbol
\end{thm} 

\begin{cor} Let $\bomega_{j,b}\bomega\in\bp\bor_\xi$ with $j<k=\min\bomega$. We have $$\dim_{\bc(q)}[\bomega_{j+1,\xi(j)}\bomega] =\dim M(0, \omega_{j+1}+\wt\bomega).$$\end{cor}
\begin{pf} If $j+1\ne k$ then  $\bomega_{j+1,\xi(j)}\bomega\in\bp\bor_\xi$ and  the corollary is immediate from Theorem \ref{bcm}(ii).
Suppose that  $j+1=k$. If $\bomega_{j+1,\xi(j)}\bomega=\bof_{j+1}$ then the assertion of the corollary is just  Theorem \ref{bcm}(iii).
Otherwise
$$\bomega_{j+1,\xi(j)}\bomega=\bof_{j+1}\bomega',\ \ \bomega'\in\bp\bor_\xi,\ \  \wt\bomega'=\wt\bomega-\omega_{j+1}.$$
  Theorem \ref{irredcrit}(a) gives
$[\bomega_{j+1,\xi(j)}\bomega]=[\bof_{j+1}][\bomega']$. Together with 
  parts (ii) and (iii) of Theorem \ref{dimprod}  we get  $$\dim_{\bc(q)}[\bomega_{j+1,\xi(j)}\bomega]= \dim_{\bc(q)}[\bof_{j+1}]\dim_{\bc(q)}[\bomega']=\dim M(0,2\omega_{j+1})\dim M(0,\wt\bomega-\omega_{j+1}).$$ Now using part (i) of the theorem we see that the right hand side is $\dim M(0,\omega_{j+1}+\wt\bomega)$ and the corollary is established.
\end{pf}
Along with  Section \ref{lb} we have now established  the following inequality. Let $\bomega_{i,a}\bomega_{j,b}\bomega\in\bp\bor_\xi$ with $i<j<\min\bomega$. Then
\begin{multline}\label{geqxi} \dim M(\omega_i,0)\dim M(0,\omega_j+\wt\bomega)\ge  \dim M(0,\wt\bomega+\omega_i+\omega_j) \\  + \dim M(0,\omega_{i-1})\dim M(0,\omega_{j+1}+\wt\bomega).\end{multline}
and  Proposition \ref{csesfa} follows if we prove that the preceding inequality  is actually an equality. This is done in the rest of the section.

 \subsection{} We deduce a consequence of the preceding discussion.
\begin{lem}\label{geq}
Let $\lambda_0\in P^+$, $\lambda_1\in P^+(1)$, $\lambda=2\lambda_0+\lambda_1$ and $1\le i<i_1=\min\lambda_1$.
Then \begin{eqnarray*}\notag&\dim M(\omega_i,0)\dim M(0,\lambda)&\ge  \dim M(0,\lambda+\omega_i) + \dim M(0,\omega_{i-1})\dim M(0,\omega_{i_1+1}+\lambda-\omega_{i_1}).\end{eqnarray*}

\end{lem}
\begin{pf} By Theorem \ref{bcm}(i) we see that for $\mu\in\{\lambda,\lambda+\omega_i, \lambda+\omega_{i_1+1}-\omega_{i_1}\}$ we can write $$\dim M(0,\mu)=\dim M(0,2\lambda_0)\dim M(0,\mu-2\lambda_0).$$Hence  the Lemma follows if we prove that 
\begin{eqnarray*}\notag&\dim M(\omega_i,0)\dim M(0,\lambda_1)&\ge  \dim M(0,\lambda_1+\omega_i) + \dim M(0,\omega_{i-1})\dim M(0,\omega_{i_1+1}+\lambda_1-\omega_{i_1}).\end{eqnarray*}
Comparing this with \eqref{geqxi} we see that it suffices to prove that we can find a height function $\xi$ such that there exists an element $\bomega_{i,a}\bpi\in\bp\bor_\xi$ with $\lambda_1=\wt\bpi$. Writing $\lambda_1=\omega_{i_1}+\cdots+\omega_{i_k}$ take  $\xi: I\to \bz_+$ such that   \begin{gather*} \xi(m)=m,\ \ 1\le m\le i_1,\ \ \ \ \xi(i_k+j) = \xi(i_k)+(-1)^kj, \ \ 1\leq j\leq n-i_k, \ \ {\rm and} \
\\ \xi(i_{j+1}) - \xi(i_j) = (-1)^{j} (i_{j+1}-i_j), \ \ 1\leq j\leq k-1.
\end{gather*} 
If $k=1$ then $\bomega_{i,i-1}\bomega_{i_1,i_1+1}\in\bp\bor_\xi$ and otherwise $\bomega_{i,i-1}\bomega(i_1,i_k)\in\bp\bor_\xi$ and  the Lemma is proved.
\end{pf}

\subsection{}   Given  a  module $V $  for $\lie {sl}_{n+1}[t]$    and $z\in\bc$ denote by $V^z$ the $\lie {sl}_{n+1}[t]$-module with underlying vector space $V$ and action given by,  $$(x\otimes t^r) w= (x\otimes (t+z)^r)w,\ \ x\in\lie {sl}_{n+1},\ \  r\in\bz_+,\  w\in V.$$  Suppose that    $V_1, V_2$ are     cyclic finite--dimensional $\lie {sl}_{n+1}[t]$-modules  with cyclic vectors  $v_1$ and $v_2$ respectively. It was proved in \cite{FL} that if $z_1,z_2$ are distinct complex numbers, then   the tensor product
$V_1^{z_1}\otimes V_2^{z_2}$ is a cyclic $\lie {sl}_{n+1}[t]$-module generated by $v_1\otimes v_2$.  Further this module    admits a filtration by the non--negative integers: the $r$-th filtered piece of $V_1^{z_1}\otimes V_2^{z_2}$ is spanned by elements of the form $(y_1\otimes  t^{s_1})\cdots (y_m\otimes t^{s_m})(v_1\otimes v_2)$ where $m\ge 0$,  $y_1,\cdots, y_m\in\lie {sl}_{n+1}$, $s_1,\cdots, s_m\in\bz_+$ and  $s_1+\cdots +s_m\le r$. 
The associated graded  space   is called a  fusion product and is denoted 
 $V_1^{z_1}* V_2^{z_2}$. It admits a canonical    $\lie {sl}_{n+1}[t]$-module structure and 
 is  generated by  the image of  $v_1\otimes v_2$ and, $$\dim \left(V_1^{z_1}* V_2^{z_2}\right)=\dim V_1\dim V_2.$$
\begin{prop} \label{reseq} Let $\lambda_0\in P^+$, $\lambda_1\in P^+(1)$, $\lambda=2\lambda_0+\lambda_1$,  $1\le i<\min\lambda_1$ and  $z_1\ne z_2\in\bc$.  There exists a   surjective map  of $\lie{sl}_{n+1}[t]$--modules $$M(\omega_i,\lambda)\to M^{z_1}(0,\lambda)*M^{z_2}(0,\omega_{i}),\ \ \ m_{i,\lambda}\to m_{0,\lambda}* m_{0,\omega_{i}}.$$ In particular,
 $\dim M(\omega_i,\lambda)\ge \dim M(0,\lambda)\dim M(0,\omega_{i}).$\
\end{prop}

\begin{pf} The proposition follows if we prove that the  element  $\bom:= m_{0,\lambda}* m_{0,\omega_{i}}$ (which generates $M^{z_1}(0,\lambda)*M^{z_2}(0,\omega_{i})$) satisfies the same relations as $m_{i,\lambda}$.  We first prove that $\bom$ satisfies the three relations in \eqref{localweyld}. The first relation in that equation is true in  the tensor product $M^{z_1}(0,\lambda)\otimes M^{z_2}(0,\omega_{i})$ and hence hold in the fusion product as well. For the second relation, we use the definition of $M^{z_1}(0,\lambda)$ and $M^{z_2}(0,\omega_i)$ to see that \begin{gather*}(h\otimes (t-z_1)^r)(m_{0,\lambda}\otimes m_{0,\omega_{i}})=((h\otimes t^r) m_{0,\lambda})\otimes m_{0,\omega_{i}} +  m_{0,\lambda}\otimes (h\otimes (t+z_2-z_1)^r)m_{0,\omega_{i}}.\end{gather*} where the action on the right hand side is in $M(0,\lambda)\otimes M(0,\omega_i)$.
If $r=0$ the relation holds in the tensor product and we are done. If $r\geq 1$, 
the  first  term on the right hand side is zero  and
the second term is $\omega_{i}(h)(z_2-z_1)^r (m_{0,\lambda}\otimes m_{0,\omega_{i}})$. Hence$$ (h\otimes (t-z_1)^r)(m_{0,\lambda}\otimes m_{0,\omega_{i}}) \in\bu(\lie{sl}_{n+1}[t])[0],$$ and so 
in the associated graded space we get 
 $$(h\otimes t^r)\bom=(h\otimes (t-z_1)^r)\bom=0,\ \ r\ge 1.$$
The third relation in \eqref{localweyld} is  immediate  from the finite--dimensional representation theory of $\lie{sl}_{n+1}$.  Next  a straightforward calculation gives, 
 $$(x_{p}^-\otimes (t-z_1)^{(\lambda_0+\lambda_1)(h_p)}(t-z_2)^{\omega_{i}(h_p)})(m_{0,\lambda}\otimes m_{0,\omega_{i}})=0, \ \ 1\le p\le n,$$ and
$$ (x_\alpha^-\otimes (t-z_1)^{\lambda_0(h_\alpha)}(t-z_2))(m_{0,\lambda}\otimes m_{0,\omega_{i}})=0,\ \ \alpha\in R^+(\lambda_1).$$
 This means that in the fusion product we have 
 $$(x_{p}^-\otimes t^{(\lambda_0+\lambda_1)(h_p)+\omega_{i}(h_p)})\bom=0, \ \ 1\le p\le n,$$ and
$$ (x_\alpha^-\otimes t^{\lambda_0(h_\alpha)+1})\bom=0,\ \ \alpha\in R^+(\lambda_1), $$ which proves that $\bom$ satisfies the relations in \eqref{mil}. This completes the proof of the proposition.
\end{pf}

\subsection{} We deduce some additional relations satisfied by $m_{i,\lambda}$. Note that by  the second relation in \eqref{localweyld} we get for $\alpha\in R^+$, $r\in\bz_+$,
$$(x_\alpha^\pm \otimes t^{r}) m_{i,\lambda} = 0 \implies (h_\alpha\otimes t^p)(x_\alpha^\pm  \otimes t^{r} )m_{i,\lambda} = 0\implies (x_\alpha^\pm  \otimes t^{r+p} )m_{i,\lambda} =0,\ \  p \in \bz_+.$$
Together with the  first relation in \eqref{mil} we  have by a simple induction on $k-j$ that for all $1\le j\le k\le n$,
\begin{equation}\label{ij}  (x^-_{j,k}\otimes t^r) m_{i,\lambda}=0,\ \ {\rm{if}} \ \ r\ge(\lambda_0+\lambda_1 + \omega_i)(h_{j,k}).\end{equation}
Since
 $(x_{j}^-\otimes t^{\lambda_0(h_{j})+1})m_{i,\lambda}=0$, $1\le j\le n$, a simple  calculation  (see \cite{CV} for instance) shows that  $$0=  (x_{j}^+\otimes t)^{2\lambda_0(h_{j})} (x_{j}^-\otimes 1)^{2\lambda_0(h_{j})+2} m_{i,\lambda}=(x_{j}^-\otimes t^{\lambda_0(h_{j})})^2 m_{i,\lambda}.$$
 If  $\alpha_{j,k} \in R^+(\lambda_1)$ then by using the preceding two relations we get\begin{equation}\label{conseq}0=  (x^-_{j+1, k}\otimes t^{\lambda_0(h_{j+1,k})+1})(x_{j}^-\otimes t^{\lambda_0(h_{j})})^2 m_{i,\lambda}  =(x^-_{j, k}\otimes t^{\lambda_0(h_{j,k})+1})(x_{j}^-\otimes t^{\lambda_0(h_{j})})m_{i,\lambda}.
\end{equation}

\begin{prop}\label{rses} Suppose that $\lambda=2\lambda_0+\lambda_1$ with $\lambda_1\in P^+(1)$ and let $i<i_1=\min\lambda$.
There exists a right   exact sequence of $\lie {sl}_{n+1}[t]$--modules $$ M(\omega_{i-1}, \lambda-\omega_{i_1}+\omega_{i_1+1})\to M(\omega_i,\lambda)\to M(0,\lambda+\omega_{i})\to 0.$$
 \end{prop}
 \begin{pf} Set
 \begin{gather*}s=(\lambda_0+\lambda_1)(h_{i,i_1}),\ \ \ \min(\lambda_1-\omega_{i_1})=i_2,\\ \lambda_2=\lambda-\omega_{i_1}+\omega_{i_1+1}=2(\lambda_0+\delta_{i_1+1,i_2}\omega_{i_1+1})+\lambda_1-\omega_{i_1}+(1-\delta_{i_1+1,i_2})\omega_{i_1+1}.\end{gather*}In view of the short exact sequence in \eqref{ses1} it suffices to prove that the assignment $$m_{i-1,\lambda_2}\to (x_{i,i_1}^-\otimes t^{s})m_{i,\lambda}$$ extends to a well--defined morphism $M(\omega_{i-1},\lambda_2)\to M(\omega_i,\lambda)$ of $\lie g[t]$--modules. In other words it is enough to check that the element  $m=(x_{i,i_1}^-\otimes t^{s})m_{i,\lambda}$ satisfies the defining relations of $M(\omega_{i-1},\lambda_2)$. This is a tedious but straightforward checking. The first thing to check is that $m$ satisfies the defining relations of $W_{\loc}(\lambda_2+\omega_{i-1})$.  For this, we observe that for $1
 \le j\le n$, \begin{eqnarray*}&(x_j^+\otimes 1)m&= [(x_j^+\otimes 1),(x_{{i,i_1}}^-\otimes t^{s})] m_{i,\lambda}= (A\delta_{j,i}(x^-_{{i+1,i_1}}\otimes t^{s})+B\delta_{j,i_1}(x^-_{{i,i_1-1}}\otimes t^{s}))m_{i,\lambda},\end{eqnarray*} for some $A,B\in\bc$. It follows from \eqref{ij} that the right hand side is zero  once we note that
 $$s=(\lambda_0+\lambda_1)(h_{i,i_1})\ge 
 \max\left\{  (\lambda_0+\lambda_1)(h_{i+1,i_1}), \ \  (\lambda_0+\lambda_1)(h_{i,i_1-1})\right\}.$$
  For the second relation in \eqref{localweyld}  we observe   $$(h\otimes t^r)m = [h\otimes t^r, x_{i,i_1}^-\otimes t^{s}]m_{i,\lambda}= -(\delta_{r,0}\lambda-\alpha_{i,i_1})(h)(x^-_{{i,i_1}}\otimes t^{s+r})m_{i,\lambda}. $$ If $r\ge 1$ then $s+r\ge(\lambda_0+\lambda_1+\omega_i)(h_{i,i_1})$ and hence the right hand side is zero by \eqref{ij}.
   The final relation in  \eqref{localweyld} holds    by the standard representation theory of $\lie{sl}_{n+1}$. Next we check that $m$ satisfies the relations in \eqref{mil}. We first show that  \begin{gather*} (x_p^-\otimes t^{r_p})m=0,\ \ 
r_p = (\lambda_0+\lambda_1 - \omega_{i_1}+\omega_{i_1+1} + \omega_{i-1})(h_p), \ \ 1\le p\le n.\end{gather*}
If $p\in\{i,i_1\}$ this follows from \eqref{conseq}.
 Assume that $p\notin\{i,i_1\}$.
 Then
 $$(x^-_p\otimes t^{r_p})m=\begin{cases} (x_{{i,i_1}}^-\otimes t^{s})(x_p^-\otimes  t^{r_p})m_{i,\lambda},\ \ p\notin\{i-1,i_1+1\},\\
  (x_{{i-1,i_1}}^-\otimes t^{s+r_{i-1}})m_{i,\lambda},\ \ p=i-1,\\
   (x_{{p,i_1+1}}^-\otimes t^{s+r_{i_1+1}})m_{i,\lambda},\ \ p=i_1+1<i_2,\\
  [x^-_{i,i_1-1}\otimes t^{s}, x^-_{i_1,i_1+1}\otimes t^{r_{i_1+1}}]m_{i,\lambda},\ \ p=i_1+1=i_2.
 \end{cases} $$
 If $p\notin\{i-1,i_1+1\}$ then $r_p=(\lambda_0+\lambda_1+\omega_i)(h_p)$ and $(x^-_p\otimes t^{r_p})m_{i,\lambda}=0$.
If $p=i-1$, then $$s+r_{i-1}=(\lambda_0+\lambda_1)(h_{i,i_1})+
(\lambda_0+\lambda_1+\omega_{i-1})(h_{i-1})
= (\lambda_0+\lambda_1+\omega_i)(h_{i-1,i_1})$$ and \eqref{ij} gives $(x_{{i-1,i_1}}^-\otimes t^{s+r_{i-1}})m_{i,\lambda}=0$.
  If $p=i_1+1$ and $i_1+1\neq i_2$ a similar argument shows that $ (x_{{i,i_1+1}}^-\otimes t^{s+r_{i_1+1}})m_{i,\lambda}=0$. 
  If $p=i_1+1=i_2$, then one checks $$(x^-_{i,i_1-1}\otimes t^{s})m_{i,\lambda}=0= (x^-_{i_1,i_1+1}\otimes t^{r_{i_1+1}})m_{i,\lambda}.$$ In all cases  the first relation in  \eqref{mil} is now established.  
The  second relation in \eqref{mil} follows if we prove that  
$$ (x_{\alpha}^-\otimes t^{(\lambda_0+\delta_{i_1+1,i_2}\omega_{i_1+1})(h_{\alpha})+1})(x_{{i,i_1}}^-\otimes t^{s})m_{i,\lambda} = 0,\ \ \alpha\in R^+(\lambda_1-\omega_{i_1}+(1-2\delta_{i_1+1,i_2})\omega_{i_1+1}).$$ If $i_1+1=i_2$, then $$ R^+(\lambda_1-\omega_{i_1}+(1-2\delta_{i_1+1,i_2})\omega_{i_1+1})\subset R^+(\lambda_1)-\{\alpha_{i_1,i_1+1}\}
$$ and if $i_1+1<i_2$ then 
$$R^+(\lambda_1-\omega_{i_1}+(1-2\delta_{i_1+1,i_2})\omega_{i_1+1})= (R^+(\lambda_1)-\{\alpha_{i_1,i_2}\})\cup \{\alpha_{i_1+1,i_2}\}.$$
If
 $\alpha\neq \alpha_{i_1+1,i_2}$ then  $[x_{\alpha}\otimes t^r,x_{i,i_1}\otimes t^s]= 0$, for each $r\in \bz_+$ and hence we get
$$(x_{\alpha}^-\otimes t^{\lambda_0(h_{\alpha})+1})(x_{i,i_1}^-\otimes t^{s})m_{i,\lambda}= (x_{i,i_1}^-\otimes t^{s})(x_{\alpha}^-\otimes t^{\lambda_0(h_{\alpha})+1})m_{i,\lambda} = 0.$$
If $\alpha = \alpha_{i_1+1,i_2}$ then  $i_1+1<i_2$ and so by the defining relations of $M(\omega_i,\lambda)$ we have
$$(x^-_{i_1,i_2}\otimes t^{\lambda_0(h_{i_1,i_2})+1})m_{i,\lambda}= 0 = (x^-_{i,i_1-1}\otimes t^{\lambda_0(h_{i, i_1-1})+1})m_{i,\lambda}=0,$$ and so $$(x_{i,i_2}^-\otimes t^{\lambda_0(h_{i,i_2})+2})m_{i,\lambda}=0.$$ It follows that
 $$(x_{i_1+1,  i_2}^-\otimes t^{\lambda_0(h_{i_1+1,i_2})+1})(x^-_{i,i_1}\otimes t^{s})m_{i,\lambda} = A(x_{i,i_2}^-\otimes t^{\lambda_0(h_{i,i_2})+2})m_{i,\lambda}=0,\ \ A\in\bc,$$
which completes the proof of \eqref{mil} and so also of the Proposition.
 \end{pf}

\subsection{} The proof of Proposition \ref{csesfa} is completed in the course of establishing the following claim: for  $\lambda=2\lambda_0+\lambda_1\in P^+$ and $i<\min\lambda_1$, we have \begin{gather}\label{mfusion}\dim M(\omega_i,\lambda)=\dim M(\omega_i,0)\dim M(0,\lambda).\end{gather} 
The claim is  proved  by an induction on $i$. Induction begins at  $i=0$ when there is nothing to prove since $M(0,0)\cong \bc$. Otherwise using Proposition \ref{rses} we have $$\dim M(\omega_i,\lambda)\le \dim M(\omega_{i-1},\lambda-\omega_{i_1}+\omega_{i_1+1})+\dim M(0,\lambda+\omega_i).$$ The following equality is clear if $i=1$, 
and otherwise holds by the inductive hypothesis,
$$\dim M(\omega_{i-1},\lambda-\omega_{i_1}+\omega_{i_1+1})=\dim M(0,\omega_{i-1})\dim M(0,\lambda-\omega_{i_1}+\omega_{i_1+1}),$$  and hence $$\dim M(\omega_i,\lambda)\le \dim M(0,\omega_{i-1})\dim M(0,\lambda-\omega_{i_1}+\omega_{i_1+1})+\dim M(0,\lambda+\omega_i).$$ By Proposition \ref{reseq} we have  $\dim M(\omega_i,0)\dim M(0,\lambda)\le \dim M(\omega_i,\lambda)$ and hecne we get 
$$\dim M(\omega_i,0)\dim M(0,\lambda)\le  \dim M(\omega_{i-1})\dim M(0,\lambda-\omega_{i_1}+\omega_{i_1+1})+\dim M(0,\lambda+\omega_i).$$ Lemma \ref{geq} now shows that all the inequalities are actually equalities and the proof of the inductive step is complete. Notice that we have also proved that the inequality in \eqref{geqxi} is an equality and so the proof of Proposition \ref{csesfa} is also complete.


\begin{thebibliography}{1}


\bibitem{BP} F. Borges, Pierin, T. C. {\emph A new cluster character with coefficients for cluster category}, arXiv:1804.10606







\bibitem{BN} J.~Beck and H.~Nakajima, {\em Crystal bases and two--sided cells of quantum affine algebras}, Duke Math.
J. {\bf 123} (2004), no. 2, 335--402.


\bibitem{BCM} M. Brito, V. Chari and A. Moura, {\em Demazure modules of level two  and prime representations of  quantum affine $\lie{sl}_{n+1}$}, J. Inst. Math. Jussieu (2015), 31 pages.


\bibitem{BFZ} A.~Berenstein, S.~Fomin and A.~ Zelevinsky, {\em Cluster algebras III: upper bounds and double Bruhat cells}, Duke Math.
J. 126 (2005), 1–52.

\bibitem{CC} P. Caldero and F. Chapoton, Cluster algebras as Hall algebras of quiver representations, Commentarii Mathematici Helvetici 81 (2006), 596–616.

\bibitem{CBraid}
V.~Chari, {\em Braid group actions and tensor products}, Int. Math. Res. Notices (2002), 357--382.


\bibitem{CMY}V. Chari, A. Moura and C. Young, {\em Prime representations from a
homological perspective}, Math. Z. {\bf 274} (2013), 613--645.

\bibitem{CPqa} V.~Chari and A.~Pressley, {\it Quantum affine algebras}, Comm. Math. Phys. {\bf 142} (1991), 261--283.

\bibitem{CPminaff} V.~Chari and A.~Pressley, {\em Minimal affinizations of representations of quantum
groups: the nonsimply laced case}, Lett. Math. Phys. {\bf 35} (1995), 99--114.

\bibitem{CPminaffsl} V.~Chari and A.~Pressley, {\em Minimal affinizations of representations of quantum groups:
the simply--laced case}, J. Alg. {\bf 184} (1996), no. 1, 1-30.


\bibitem{CSVW}
V.~Chari, P.~Shereen, R.~Venkatesh and J.~Wand, \emph{A Steinberg type decomposition theorem for higher level Demazure modules}, arXiv:1408.4090.

\bibitem{CV} 
V.~Chari, R.~Venkatesh, \emph{Demazure modules, Fusion Products, and $Q$-systems}, Commun. Math. Phys. {\bf 333} (2015), no. 2, 799-830.

\bibitem{Dup9} G. Dupont, Quantized Chebyshev polynomials and cluster characters with coefficients, Journal of Algebraic Combinatorics 31 (2010), 501–532, 10.1007/s10801-009-0198-8.

\bibitem{FK10} C. Fu and B. Keller, On cluster algebras with coefficients and 2-Calabi-Yau categories, Trans. Amer. Math. Soc. 362 (2010), 859–895.



\bibitem{FL} B. Feigin and S. Loktev, {\em On Generalized Kostka Polynomials and the Quantum Verlinde Rule}Differential topology,infinite--di\-men\-si\-onal Lie algebras, and applications, Amer.Math. Soc. Transl. Ser. 2, Vol. {\bf 194} (1999), p. 61--79,math.QA/9812093.

\bibitem{FZ} S. Fomin, A. Zelevinsky, {\it Cluster algebras: notes for the CDM-03 conference}. Current developments in mathematics (2003), 1--34, Int. Press, Somerville, MA, 2003.

\bibitem{FoL} G.~Fourier, P.~Littelmann. \emph{Weyl modules, Demazure modules, KR-modules, crystals, fusion products, and limit constructions.} Adv. Math. \textbf{211} (2007), no.2, 566-593.

\bibitem{FM}
E.~Frenkel and E.~Mukhin, {\it Combinatorics of $q$-character of finite dimensional representations of quantum affine algebras}, Commun. Math. Phys. {\bf 216} (2001), 23--57.

\bibitem{FR}
E. Frenkel and N. Reshetikhin, {\it The $q$-characters of representations of quantum affine algebras and deformations of $W$-algebras}, Contem. Math. {\bf 248} (1998), 163--205.

\bibitem{H1} D.~Hernandez, {\em Algebraic approach to $q,t$-characters}, Adv. Math. {\bf 187} (2004), 1--52.

\bibitem{H2} D.~Hernandez, {\em Monomials of $q$ and $q,t$-characters for non simply--laced quantum affinizations}, Math. Z. {\bf 250} (2005),
443--473.

\bibitem{H} D.~Hernandez, {\it The Kirillov--Reshetikhin conjecture and solutions of $T$-systems}, J. Reine Angew. Math. {\bf 2006} (2006), 63--87

\bibitem{H4} D.~Hernandez, {\em On minimal affinizations of representations of quantum groups}, Comm. Math. Phys. {\bf 277} (2007), 221--
259.

\bibitem{H5} D.~Hernandez, {\em Simple tensor product}, Invent. Math {\bf 181} (2010), no. 3, 649--675.

\bibitem{H6} D.~Hernandez, {\em Cyclicity and R-matrices},  arXiv:1710.05756.

\bibitem{HL1}
D. Hernandez and B. Leclerc, {\em Cluster algebras and quantum affine algebras}, Duke Math. J. {\bf 154} (2010), 265--341,
DOI 10.1215/00127094-2010-040.

\bibitem{hl:hall}
D.~Hernandez and B.~Leclerc, {\em Quantum Grothendieck rings and derived Hall algebras}, J. Reine Angew. Math. {\bf 701} (2015), 77--126.

\bibitem{HL2}
D.~Hernandez and B.~Leclerc, {\em Monoidal categorifications of cluster algebras of type A and D}, Symmetries, Integrable Systems and Representations, Springer Proceedings in Mathematics \& Statistics 40 (2013), 175--193.

\bibitem{HL13}
D.~Hernandez and B.~Leclerc, {\em A cluster algebra approach to q--characters of Kirillov--Reshetikhin module}, J. Eur. Math. Soc. {\bf 18} (2016), no. 5, 1113--1159.

\bibitem{KKKO} S--J.~Kang, M.~Kashiwara, M.~Kim and S--J.~Oh, {\em Monoidal categorification of cluster algebras},
Preprint arXiv:1412.8106

\bibitem{KKKO2} S--J.~Kang, M.~Kashiwara, M.~Kim and S--J.~Oh, {\em Monoidal categorification of cluster algebras
II}, Preprint arXiv:1502.06714.

\bibitem{MRZ} G. Muller, J. Rajchot, AND B. Zykoski, Lower bound cluster algebras,: presentations, Cohen--Macaulayness and  normality, arxiv 1508.02314.pdf.

\bibitem{N1}H.~Nakajima, {\em $t$-analogs of $q$-characters of quantum affine algebras of type $A_n$, $D_n$}, in Combinatorial and geometric
 representation theory (Seoul, 2001), 141--160, Contemp. Math., {\bf 325}, Amer. Math. Soc., Providence, RI (2003).

\bibitem{N2}H.~Nakajima, {\em Quiver Varieties and $t$--Analogs of $q$--Characters of Quantum Affine Algebras},
Ann. of Math. {\bf 160}, 1057--1097 (2004).


\bibitem{Nak} H.~Nakajima, {\em Quiver varieties and cluster algebras}, Kyoto J. Math. {\bf 51} (2011), 71--126.

\bibitem{Naoi} K.~Naoi, \emph{Weyl Modules, Demazure modules and finite crystals for non-simply laced type}, Adv. Math. \textbf{229} (2012), no.2, 875--934.

\bibitem{VV} M.~Varagnolo and E.~Vasserot, {\em Standard modules of quantum affine algebras}, Duke Math. J. {\bf 111}, no.
3, 509--533 (2002).

\bibitem{Wand} J.~Wand, {\em Demazure flags of local Weyl modules},  Ph.D. Thesis, UC Riverside,  \href{ http://escholarship.org/uc/item/8nr0m842}{http://escholarship.org/uc/item/8nr0m842}  
\end{thebibliography}
\end{document}